\documentclass{amsart}

\usepackage{amsmath}
\usepackage{amssymb}
\usepackage[all]{xy}
\usepackage{picinpar}
\usepackage[all]{xy}

\DeclareMathOperator{\grado}{deg}
\DeclareMathOperator{\Gal}{Gal}

\DeclareMathOperator{\cogalois}{cog}

\DeclareMathOperator{\caracteristica}{char}

\numberwithin{equation}{section}
\newtheorem{teorema}{Theorem}[section]
\newtheorem{ejemplo}[teorema]{Example}
\newtheorem{lema}[teorema]{Lemma}
\newtheorem{proposicion}[teorema]{Proposition}
\newtheorem{corolario}[teorema]{Corollary}
\newtheorem{comentario}[teorema]{Remark}
\newtheorem{definicion}[teorema]{Definition}

\title[Radical Extensions for the Carlitz--Hayes Module]
{Radical Extensions for the Carlitz--Hayes Module}

\author[Marco S\'anchez]{Marco S\'anchez--Mirafuentes}
\address{Universidad Aut\'onoma Metropolitana\\
Unidad Iztapalapa\\
Divisi\'on de Ciencias B\'a\-si\-cas e Ingenier\'ia\\
Departamento de Matem\'aticas\\
Av. San Rafael Atlixco 186, Col Vicentina
Del. Iztapalapa,\\
C.P. 09240 M\'exico, D.F., M\'exico}
\email{kmasm1969@yahoo.com.mx}

\author[Gabriel Villa]
{Gabriel Villa--Salvador}
\address{
Departamento de Control Autom\'atico\\
Centro de Investigaci\'on y de Estudios Avanzados del I.P.N.\\
Apartado Postal 14--740, CP 07000, M\'exico, D.F., M\'exico.}
\email{gvillasalvador@gmail.com, gvilla@ctrl.cinvestav.mx}



\date{July 17th, 2013}

\begin{document}

\begin{abstract}
Let $L/K$ be a finite extension of congruence function fields. We
say that $L/K$ is a {\it radical extension} if $L$ is generated by
roots of polynomials $u^{M}-\alpha \in K[u]$, where $u^{M}$ is the
action of Carlitz-Hayes. We study a special class of these
extensions, the {\it radical cyclotomic} extensions. We prove that
any radical cyclotomic extension has order a power of the
characteristic of $K$. We also give bounds for the Carlitz-Hayes
torsion of these extensions.
\end{abstract}

\maketitle

\section{Introduction}\label{seccion_intro}
Let $L^{\prime}/K^{\prime}$ be a field extension. In the study of
radical extensions we have the torsion and the cogalois classical groups
\begin{gather*}
T(L^{\prime}/K^{\prime})=\{u\in (L^{\prime})^\ast\mid\ u^{n}\in
K^{\prime} \text{ for some integer $n$} \}\\
\intertext{and}
\cogalois(L^{\prime}/K^{\prime})= T(L^{\prime}/K^{\prime})/(K^{\prime})^{\ast}.
\end{gather*}

There exists an important class of field extensions studied
by Greither and Harrison \cite{greither} which is called
{\em cogalois extensions}.
We say that $L^{\prime}/K^{\prime}$ is a cogalois extension if:

(1) $L^{\prime} = K^{\prime}(T(L^{\prime}/K^{\prime}))$ and

(2) ${\rm card}({\rm cog}(L^{\prime}/K^{\prime}))\leq
[L^{\prime}:K^{\prime}]$,

\noindent see \cite{greither} and \cite{barrera1}.

The analogy of number fields with congruence function fields and,
more precisely, of cyclotomic number fields with cyclotomic function
fields, leads to the natural question if there exists the analogous of the
usual torsion group with the torsion module defined by the Carlitz--Hayes' action.

We give a new definition of radical extension by using the action of
Carlitz--Hayes. A funcion field extension $L/K$ will be called {\em radical},
if $L$ can be generated by some elements $u$ with $u^{M_u}\in K$
over $K$, where $M_u$ are some rational polynomials. 
 Among these extensions we study
 {\em radical cyclotomic} extensions. An extension
is called radical cyclotomic, if it is radical, separable and {\em pure}.
They can be viewed as the generalizations of Carlitz--Kummer
extensions. These
extensions have analogous properties to those of cogalois extensions
defined in \cite{greither}, see Sections
\ref{prop_extensiones_radicales} and \ref{rad_ciclotomicas_props}. A
radical cyclotomic extension $L/K$ satisfies $L=K(T(L/K))$. Observe
the analogy with the previous definition.

In this paper we study the torsion given by the action of Carlitz--Hayes.
Thus we understand ``radical'' in the sense of this action. We study
the structure of congruence function fields generated by torsion.
In Section \ref{radicales_ciclotomicas} we define the concept of radical
cyclotomic extension as an analogue of cogalois extensions in the
classical case. We give examples of radical and nonradical cyclotomic
extensions and of pure and nonpure extensions and show that, as
in the classical case, the extension $k(\Lambda_{P^n})/k(\Lambda_P)$
is pure, where $P\in R_T$ is an irreducible polynomial and $n\in
{\mathbb N}$. In Sections \ref{prop_extensiones_radicales} and
\ref{rad_ciclotomicas_props} we give some properties of radical
and radical cyclotomic extensions and prove, as in the classical
case, that for Galois extensions, the cogalois group is isomorphic
with the group of crossed homomorphisms. 

In Section
\ref{teoremas_de_estructura} we obtain our main results: we
characterize the finite radical cyclotomic extensions.
In particular we prove that finite radical cyclotomic extensions
are $p$--extensions, where $p$ is the characteristic of the base field.
This is given in Theorems
\ref{tdim_p_galois} and \ref{tdim_p} and Corollary \ref{cogalois_grado_p_pureza}.
Examples and applications are provided in Section \ref{examples}.
Finally, in Section \ref{estimacion_para_cogalois} we find an upper
bound for the cogalois group of a radical cyclotomic extension.

\section{Notation}
We shall use the following notation.

$p$ denotes a prime number.

$q = p^{\nu}$, $\nu \in \mathbb{N}$.

$k= {\mathbb{F}}_{q}(T)$ denotes the field of rational functions.

$R_{T}={\mathbb{F}}_{q}[T]$.

$\mu(K)$ denotes the set of Carlitz roots contained in a field $K$.

$\overline{k}$ denotes an algebraic closure of $k$.

$\caracteristica(L)$ denotes the characteristic of a field $L$.

If $E/L$ is a field extension such that $k\subseteq L\subseteq
E\subseteq \overline{k}$, we denote by $T(E/L)$ the set $\{u\in
E\mid \text{there exists $M\in R_{T}$ such that $u^{M}\in L$}\}$.

$C_{m}$ denotes the cyclic group of order $m$.

\section{Radical
cyclotomic extensions}\label{radicales_ciclotomicas}

In this work we consider extensions $L/K$ such that
$k\subseteq K\subseteq L\subseteq
\overline{k}$.

In what follows we will use the Carlitz-Hayes action. Let
$\varphi:\overline{k}\rightarrow
\overline{k}$ be the Frobenius automorphism,
$\varphi(u) = u^{q}$, and let $\mu_{T}$ be the homomorphism
$\mu_{T}:\overline{k}\rightarrow\overline{k}$,
given by $\mu_{T}(u)= Tu$. We have the Carlitz-Hayes action of
$R_T$ in $\overline{k}$ given as follows:
if $M\in R_T$ and $u\in \overline{k}$,
then $u^{M}:=M(\varphi+\mu_{T})(u)$.

Now, if $M\in R_T\setminus\{0\}$, the $M$-torsion set of
$\overline{k}$, denoted by $\Lambda_{M}$, is defined
as $\Lambda_{M}:=\{u\in\overline{k}\mid u^{M}=0\}$. We also call
$\Lambda_{M}$ the set of Carlitz {\it $M$-roots}. Note that if
$a\in\overline{k}$ then set of all roots of the
polynomial $z^{M}-a$ is
$\{\alpha+\lambda\mid\lambda\in\Lambda_{M}\}$, where $\alpha$ is any
fixed root of $z^{M}-a$ in $\overline{k}$.

\begin{comentario}
{\rm{The module $\Lambda_{M}$ is analogous to the $n$-torsion
defined over $\overline{{\mathbb{Q}}}^{*}$. Since the $n$-torsion is
a cyclic ${\mathbb{Z}}$-module because it consists of roots of
unity, it can be shown in an analogous way, that $\Lambda_{M}$ is a
cyclic $R_T$-module, see \cite{villa_salvador} Chapter
12. We denote by $\lambda_{M}$ a generator of $\Lambda_{M}$ and say
that $\lambda_{M}$ is a {\it primitive Carlitz root}.}}
\end{comentario}

\begin{definicion}\label{definicion_torsion_carlitz3}
{\rm{If $M\in R_T$, $M\neq 0$, $\Phi(M)$ denotes the
number of elements of $(R_T/(M))^{*}$, the
multiplicative subgroup of $R_T/(M)$.}}
\end{definicion}

Therefore we have $\Phi(M) = {\rm card}(\{N+(M)\mid (N)+(M) =
R_T\})$. Thus
$\Phi$ is the analogous of the Euler $\varphi$ function.

In the case of number fields, if ${\mathbb{Q}}(\zeta_{n})$ is the
cyclotomic field, ${\mathbb{Q}}(\zeta_{n})/{\mathbb{Q}}$ is a Galois
extension, with Galois group $({\mathbb{Z}}/n{\mathbb{Z}})^{*}$. The
analogue for function fields is:

\begin{teorema}\label{ciclotomico_carlitz}
If $M\in R_T$, then
$k(\Lambda_{M})/k$ is a Galois extension
of degree $\Phi(M)$ and the Galois group is isomorphic to
$(R_T/M)^{*}$. In particular
$k(\Lambda_{M})/k$ is an abelian
extension. \hfill $\square$
\end{teorema}

See \cite{villa_salvador} Chapter 12, for more details.

First we see that as in the classical case, the Carlitz roots of
a cyclotomic function fields behave as the usual roots
of unity in a cyclotomic number field.

\begin{proposicion}\label{finitud_RC}
Let $M\in R_{T}$ be non-constant. Consider the extension
$k(\Lambda_{M})/k$. Then $\mu(k(\Lambda_{M})) = \Lambda_{M}$.
\end{proposicion}
\begin{proof}
Let $\mu(k(\Lambda_{M}))=\Lambda_{N}$. Since
$\Lambda_{M}\subseteq\mu(k(\Lambda_{M}))$, if for an irreducible
polynomial $P\in R_{T}$ and for $\alpha\in{\mathbb{Z}}$ we have
$P^{\alpha}\mid M$, then $P^{\alpha}\mid N$. If $P^{\alpha+1}\nmid
M$, we can not have $P^{\alpha+1}\mid N$ since otherwise the
ramification index of $P$ in $k(\Lambda_{M})/k$ would be divided by
$\Phi(P^{\alpha+1})=[k(\Lambda_{P^{\alpha+1}}):k]$ but the
ramification index in $k(\Lambda_{M})/k$ is indeed
$\Phi(P^{\alpha})$. Thus $N=M$.
\end{proof}

In what follows, unless otherwise specified, we consider extensions
$L/K$ such that $k\subseteq K\subseteq L\subseteq \overline{k}$ and
$L/k$ is finite. Moreover the previous extensions have a natural
$R_{T}$-module structure using the Carlitz-Hayes action defined
above. The first object associated with the extension $L/K$ is the
torsion group of the Carlitz action, which we denote as in the classical case:
\[
T(L/K)=\{u\in L\mid\ u^{M}\in
K \text{ for some $M\in R_{T}$} \}.
\]

Note that $T(L/K)\subseteq L$ is a subgroup of the {\it additive
group} $L$. On the other hand $T(L/K)$ is an $R_{T}$-module and the
$R_{T}$-module $T(L/K)/K$ is an $R_{T}$-torsion module. This module
is denoted by $\cogalois(L/K)$. We have $\cogalois(L/K)$ is
analogous to the group $T(L/K)/K^{*}$ in the case of an extension of
number fields $L/K$ and where $T(L/K)$ denotes the usual torsion
group, see \cite{barrera1} p.2.

\begin{definicion}\label{definicion1_c2}
{\rm{We say that an extension $L/K$ is {\it radical} if there is a
subset $A\subseteq T(L/K)$ such that $L= K(A)$. We say that $L/K$ is
{\it pure} if for every irreducible monic polynomial $M\in R_{T}$
and each $u\in L$ such that $u^{M} = 0$ we have $u\in K$. Finally,
we say that $L/K$ is a {\it radical cyclotomic} extension if it is:

(1) radical,

(2) separable and

(3) pure.

The module $\cogalois(L/K)= T(L/K)/K$ will be called {\it the
cogalois module of the extension} $L/K$.}}
\end{definicion}

\begin{ejemplo}\label{ejemplo1}
{\rm{The extension $k(\Lambda_{M})/k$, where $M\in R_{T}$, is
radical since there exists $W = \Lambda_{M}\subseteq
T(k(\Lambda_{M})/k)$ such that $k(\Lambda_{M}) = k(W)$. It is
separable, but it is not pure, since by Proposition \ref{finitud_RC}
the Carlitz roots in $k(\Lambda_{M})$ are the elements of
$\Lambda_{M}$. If $Q$ is
an irreducible factor of $M$ we have
$\lambda_{Q}\in k(\Lambda_{M})$ but
$\lambda_{Q}\notin k$. Therefore $k(\Lambda_{M})/k$ is not a radical
cyclotomic extension.}}
\hfill $\square$ 
\end{ejemplo}

The following example shows the existence of radical cyclotomic
extensions.
\begin{ejemplo}\label{ejemplo4}
{\rm{Let $p$ be an odd prime, $q=p$ and $M = T$. Consider the
extension $k(\Lambda_{M})/k$. The degree of the extension is $q -
1=p-1$. We have seen that $k(\Lambda_{M})/k$ is not pure, see
Example \ref{ejemplo1}. Now consider the polynomial $F(X) = X^{T} -
1 = X^{p} + X T - 1$.

We claim that $1\in k(\Lambda_{M})\setminus k(\Lambda_{M})^{M}$,
since otherwise, there exists $u\in k(\Lambda_{M})$ such that $u^{M}
= 1$. Let $\alpha$ be a generator of $\Lambda_{M}$. Note that
$[k(\alpha):k]=p-1$. Therefore $\{1,\alpha,\alpha^{2},\ldots,
\alpha^{p-2}\}$ is a basis of $k(\Lambda_{M})$ over $k$.

Hence, $u$ can be written as $u = a_{0} + a_{1} \alpha + \cdots+
a_{p-2} \alpha^{p - 2}$ with $a_{0},a_{1},\ldots, a_{p - 2}\in k$.
Therefore

\begin{equation}\label{ec5_1}
\begin{split}
 ~~ \quad ~~
 u^{T} &=  a^{T}_{0} + (a_{1} \alpha)^{T} + \cdots +(a_{p-2} \alpha^{p
- 2})^{T}\\
&= (a^{p}_{0} + a_{0} T) + (a^{p}_{1} \alpha^{p} + a_{1} \alpha T) +
\cdots+ (a^{p}_{p-2} \alpha^{p (p-2)} + a_{p-2} \alpha^{p-2} T).
\end{split}
\end{equation}

Since $\alpha^{T} = \alpha^{p} + \alpha T = 0$ we have $\alpha^{p} =
- \alpha T$. Therefore, since $u^{T}=1$, from equation
(\ref{ec5_1}) we obtain
\begin{align*}
1 &= (a^{p}_{0} + a_{0} T) + (a^{p}_{1} \alpha^{p} + a_{1} \alpha T)+
\cdots + (a^{p}_{p-2} \alpha^{p (p-2)} + a_{p-2} \alpha^{p-2} T)\\
&= (a^{p}_{0} + a_{0} T)+(-a^{p}_{1}\alpha T+a_{1}\alpha T) +\cdots
+ (-a^{p}_{p-2}\alpha^{p-2}T^{p-2}+ a_{p-2} \alpha^{p-2} T),
\end{align*}
that is,
\[
0=(a^{p}_{0} + a_{0} T-1)+ c_{1}\alpha+ c_{2}\alpha^{2}+\cdots + c_{p-2}\alpha^{p-2}
\]
where $c_{i}= (-1)^{i}a^{p}_{i}T^{i}+ a_{i}T$,
$i=1,\ldots, p-2$, belongs to $k$.

Hence we obtain the equation
\begin{equation}\label{Eq1N}
0 = a^{p}_{0} + a_{0} T  - 1
\end{equation}
because $\{1,\alpha,\alpha^{2},\ldots, \alpha^{p - 2}\}$ is
a basis of $k(\Lambda_{M})$ over $k$. In particular $a_0
\neq 0$. Let $a_{0} =
\frac{f(T)}{g(T)}$, with $f(T),g(T)\in R_T$ and
$(f(T),g(T)) = 1$, so we deduce the
equation $f^{p}(T) + f(T)g^{p - 1}(T) T = g^{p}(T)$. It
follows that $f(T),g(T)\in {\mathbb F}_q^{\ast}$ and
$a_0\in {\mathbb F}_q^{\ast}$ which contradicts (\ref{Eq1N}).

Let $L$ be the splitting field of $F(X)$ over $k(\Lambda_{M})$.
Then the extension $L/k(\Lambda_{M})$ is separable. From Proposition
2.3 of \cite{Schultheis}, we have that $[L:k(\Lambda_{M})] = p^{t}$,
where $t \geq 1$. If $\beta$ is a root of $F(X)$ then $L =
k(\Lambda_{M})(\beta)$, so the extension $L/k(\Lambda_{M})$ is
radical. Note that since the minimal polynomial of $\beta$
divides $F(X) = X^{p} + X T - 1$, such minimal polynomial is
$F(X)$. In particular it follows that $t = 1$.

To show that the extension $L/k(\Lambda_{M})$ is radical cyclotomic,
it remains to prove that the extension $L/k(\Lambda_{M})$ is pure.
 For this
purpose we consider a monic irreducible polynomial $N$ such that the
degree of $N$ is greater than one. Let $u\in L$ be such that $u^{N}
= 0$. We claim that $u=0$ since otherwise, from Proposition 12.2.21
of \cite{villa_salvador}, Chapter 12 and since $N$ is an irreducible
polynomial, $u\neq 0$ is a generator of $\Lambda_{N}$. Consider the
diagram
\[
 \xymatrix{& L\ar@{-}[dl]\ar@{-}[dr]  &\\
k(u)\ar@{-}[dr] & & k(\Lambda_{M})\ar@{-}[dl] \\
 & k&}
 \]

Now by Theorem \ref{ciclotomico_carlitz} we have that $[k(u):k] =
\Phi(N) =p^{{\deg}(N)}-1\geq p(p-1) = [L:k]$, but this contradicts
that $[k(u):k]\mid [L:k]$. Therefore $u=0\in k(\Lambda_{M})$. This
shows property (3) of Definition \ref{definicion1_c2}, for
polynomials of degree higher than 1.

It remains to show property (3) of Definition \ref{definicion1_c2},
for polynomials of degree 1. For this purpose, consider the
polynomials $T, T+1, \ldots, T+(p-1)$. It suffices to consider, for
example, $N = T+1$. Let $u\in L$ be such that $u^{T+1} = 0$ and
suppose that $u\notin k(\Lambda_{M})$, in particular $u\neq 0$.
Therefore we have ${\rm irr}(u,k(\Lambda_{M}))\mid (X^{p-1} + T+1)$.
This is a contradiction to our assumption that $\grado({\rm
irr}(u,k(\Lambda_{M})))= p$. Therefore $u\in k(\Lambda_{M})$.}}
\hfill $\square$ 
\end{ejemplo}

For the next example we need the following proposition.
\begin{proposicion}\label{pureza_carlitz}
Let $q > 2$, $P\in R_{T}$ be a monic irreducible polynomial and
$n\in {\mathbb{N}}$. Then the extension
$k(\Lambda_{P^{n}})/k(\Lambda_{P})$ is pure.
\end{proposicion}
\begin{proof}
If $\lambda_{Q}\in k(\lambda_{P^{n}})$, then $Q$ is ramified in
$k(\lambda_{P^{n}})/k$ which implies that $Q=P$, by Proposition
12.3.14 Chapter 12 of \cite{villa_salvador}. Therefore
$k(\lambda_{P^{n}})/k(\lambda_{P})$ is pure.
\end{proof}

\begin{ejemplo}\label{ejemplo5}
{\rm{The extension $k(\Lambda_{P^{n}})/k(\Lambda_{P})$ is cyclotomic
radical since it obviously is radical and separable and it is pure
by Lemma \ref{pureza_carlitz}.}}
\hfill $\square$ 
\end{ejemplo}

\section{Some properties of radical extensions}\label{prop_extensiones_radicales}

The radical extensions $L/K$, in the sense given in this work, have
properties similar to radical extensions considered in
\cite{greither} and \cite{Albu}. If $G$ is a torsion module we will
write
\[
{\mathcal{O}}_{G}=\{\text{order}(g)\mid g\in G\}.
\]

\begin{definicion}\label{modulotorsion2_2bis}
{\rm{A module $G$ is said to be {\it bounded} if $G$ is a torsion
module and the degrees of the elements of
${\mathcal{O}}_{G}\subseteq R_{T}$ form a bounded set, or
equivalently, if ${\mathcal{O}}_{G}$ is a finite set.}}
\end{definicion}

Let $A$ be a torsion $R_{T}$-module. Suppose that $A$ is a bounded
$R_{T}$-module, in the sense of Definition
\ref{modulotorsion2_2bis}. The least common multiple of the elements
of ${\mathcal{O}}_{A}$, will be called the {\it $R_{T}$-exponent of
$A$} or, if the context is clear, the {\it exponent} of $A$, and it
is denoted by $exp(A)$.

Now, let $E/F$ be a radical extension, not necessarily finite. There
exists $A\subseteq T(E/F)$ such that $E=F(A)$. We may replace $A$ by
the submodule of $E$ generated by $A$ and $F$, which will also be
denoted by $A$.

Now, $A/F$ is a torsion $R_{T}$-module, thus it makes sense to
consider ${\mathcal{O}}_{A/F}$. We say that an $R_{T}$-torsion
extension $E/F$ is a {\it bounded extension} if $A/F$ is a bounded
$R_{T}$-module. In this case, if $N= exp(A/F)$, we say that $E/F$ is
an {\it $N$ bounded} extension.

In this context we have the following proposition.

\begin{proposicion}\label{galois_torsion_raices}
Let $E/F$ be a bounded radical extension, not necessarily finite and
let $N = exp(A/F)$. Then $E/F$ is a Galois extension if and only if
$\lambda_{M}\in E$ for each $M\in {\mathcal{O}}_{A/F}$.
\end{proposicion}

\begin{proof}
Let $\alpha\in E$ of order $M\in R_{T}$. Thus we have
$\alpha^{M}=a\in F$. We consider the polynomial
$f(X)=X^{M}-a=\prod_{N_1} (X-(\alpha+\lambda^{N_1}_{M}))\in F[X]$. The
conjugates of $\alpha$ are
\[
\{\alpha+\xi_{1},\ldots ,\alpha+\xi_{s}\}
\]
for some elements $\xi_{i}\in \Lambda_{M}$.

Suppose that $E/F$ is Galois.
Let $B$ be the $R_{T}$-module generated by $\{\xi_{1},\ldots
,\xi_{s}\}$. Then $B\subseteq E$ and there exists $M^{\prime}\in
R_{T}$, dividing $M$, such that $B=\Lambda_{M^{\prime}}$. If
$M^{\prime}\neq M$, we have $\alpha^{M^{\prime}}=a^{\prime}\in F$
which is a contradiction. Therefore $M^{\prime}=M$ and
$\lambda_{M}\in E$.

Now suppose that $\lambda_{M}\in E$ for each
$M\in{\mathcal{O}}_{A/F}$. Let $u\in A$ and $M = \text{order}(u)$.
Since every conjugate of $u$ over $F$ is of the form
$u+\lambda^{N_1}_{M}\in E$, it follows that the extension $E/F$ is
normal and since $u$ is separable over $F$, $E/F$ is a Galois
extension.
\end{proof}

It is possible to find in some radical extensions $L/K$ explicitly a
primitive element belonging to $L$, as shown in the following
proposition.
\begin{proposicion}\label{primitivo_explicito}
Let $L/K$ be an extension such that $L=K(\alpha,\beta)$ and such
that there exist $M,N\in R_{T}$ with $\alpha^{M}=a$, $\beta^{N}=b$,
$a,b\in K$, $M$ and $N$ relatively prime. Then $L=K(\alpha+\beta)$,
that is, $\alpha+\beta$ is a primitive element.
\end{proposicion}

\begin{proof}
Since $\alpha+\beta\in K(\alpha,\beta)$ it follows that
$K(\alpha+\beta)\subseteq K(\alpha,\beta)$. On the other hand
$(\alpha+\beta)^{M}= \alpha^{M}+\beta^{M}= a+\beta^{M}\in
K(\alpha+\beta)$ and $(\alpha+\beta)^{N}= \alpha^{N}+\beta^{N}=
\alpha^{N}+b\in K(\alpha+\beta)$. Therefore we have
$\beta^{M},\alpha^{N}\in K(\alpha+\beta)$.

Since there exist $S_{1},S_{2}\in R_{T}$ such that
$1=MS_{1}+NS_{2}$, we have
\begin{gather*}
\alpha=\alpha^{1}=\alpha^{MS_{1}+NS_{2}}=a^{S_{1}}
+(\alpha^{N})^{S_{2}}\in K(\alpha+\beta)\\
\intertext{and}
\beta=\beta^{1}=\beta^{MS_{1}+NS_{2}}= 
(\beta^{M})^{S_{1}}+ b^{S_{2}}\in K(\alpha+\beta).
\end{gather*}

Therefore $K(\alpha,\beta)= K(\alpha+\beta)$. Further,
$(\alpha+\beta)^{MN}= (\alpha^{M})^{N}+(\beta^{N})^{M}\in K$.
\end{proof}

Note that the above argument can be generalized to extensions $L/K$,
with $L=K(\alpha_{1},\ldots , \alpha_{s})$ so that there is
$M_{i}\in R_{T}$ with $\alpha^{M_{i}}_{i}= a_{i}\in K$ and the
polynomials $M_{i}$ are pairwise relatively prime.

\section{Some properties of radical cyclotomic extensions}\label{rad_ciclotomicas_props}

Radical cyclotomic extensions have some properties analogous to the
properties of classic cogalois extensions. We first need a lemma.
\begin{lema}\label{pureza}
Let $K\subseteq L\subseteq L^{\prime}$ be a tower of fields. Then
$L^{\prime}/K$ is pure if and only if $L^{\prime}/L$ and $L/K$ are
pure.
\end{lema}
\begin{proof}
Suppose that $L^{\prime}/K$ is pure, let $\lambda_{P}\in L^{\prime}$
and $P\in R_{T}$, monic and irreducible, such that $\lambda^{P}_{P}
= 0$. Then $\lambda_{P}\in K\subseteq L$, since $L^{\prime}/K$ is
pure. Therefore $L^{\prime}/L$ is pure. Similarly it is shown that
$L/K$ is pure.

Now suppose that $L^{\prime}/L$ and $L/K$ are pure. Let
$\lambda_{P}\in L^{\prime}$ and $P\in R_{T}$ be a monic and
irreducible polynomial such that $\lambda^{P}_{P} = 0$. Since
$L^{\prime}/L$ is pure, we have $\lambda_{P}\in L$ and since $L/K$
is pure, it follows that $\lambda_{P}\in K$.
\end{proof}

\begin{proposicion}\label{prop_cog}
Let $K\subseteq L\subseteq L^{\prime}$ be a tower of fields.
\begin{itemize}
\item[(1)] The following is an exact sequence of $R_{T}$-modules
\[
0\rightarrow \cogalois(L/K)\rightarrow 
\cogalois(L^{\prime}/K)\rightarrow \cogalois(L^{\prime}/L).
\]

\item[(2)] If $L^{\prime}/K$ is a radical cyclotomic extension, then
the extension $L^{\prime}/L$ is radical cyclotomic.

\item[(3)] If the extension $L^{\prime}/K$ is radical, and the extensions $L^{\prime}/L$ and $L/K$ are radical cyclotomic,
then $L^{\prime}/K$ is radical cyclotomic.
\end{itemize}
\end{proposicion}

\begin{proof}
(1) The canonical homomorphism
\[
\cogalois(L^{\prime}/K)\rightarrow \cogalois(L^{\prime}/L),\,\, x+K\mapsto x+L
\]
is an $R_{T}$-homomorphism with kernel $\cogalois(L/K)$.
This proves that the sequence of $R_{T}$-modules
\[
0\rightarrow \cogalois(L/K)\rightarrow \cogalois(L^{\prime}/K)\rightarrow \cogalois(L^{\prime}/L)
\]
is exact.

(2) Since $L^{\prime}/K$ is separable, we have $L^{\prime}/L$ is
separable and, by Lemma \ref{pureza}, $L^{\prime}/L$ is pure.
Finally, since $T(L^{\prime}/K)\subseteq T(L^{\prime}/L)$, we have
that $L^{\prime}/L$ is radical.

(3) Since $L^{\prime}/L$ and $L/K$ are radical cyclotomic
extensions, we have that both are separable and pure. Therefore, by
Lemma \ref{pureza}, the extension $L^{\prime}/K$ is pure and
separable. Thus $L^{\prime}/K$ is a radical cyclotomic extension.
\end{proof}

We will see that 
the $R_{T}$-module $\cogalois(L/K)$ is finite for some extensions
$L/K$. To begin with, consider $L/K$ a Galois extension of function
fields, with Galois group $G = \Gal(L/K)$. Note that $\mu(L)$ is a
$G$-module, with the following action: given $\sigma\in G$ and $u\in
\mu(L)$ let $\sigma\cdot u = \sigma(u)$. Since the Carlitz-Hayes
action commutes with $\sigma$, $\sigma\cdot u$ is well defined.

\begin{definicion}\label{definicion_morfismo_cruzado}
{\rm{A function $f:G\rightarrow \mu(L)$ is said to be a {\it crossed
homomorphism} of $G$ with coefficients in $\mu(L)$ if for every
$\sigma,\tau\in G$ we have $f(\sigma\circ\tau) = f(\sigma) +
\sigma\cdot f(\tau)$.

The set of crossed homomorphisms is denoted by
\[
Z^{1}(G,\mu(L)).
\]
and $B^{1}(G,\mu(L))$ denote the subset of $Z^{1}(G,\mu(L))$, given
by $\{\chi\in Z^{1}(G,\mu(L))\mid\,\text{there exists $u\in \mu(L)$
such that $\chi=f_{u}$}\}$, where $f_{u}$ is the function defined by
\begin{equation}\label{ec7_1}
f_{u}(\sigma)=\sigma(u)-u\, \text{for each $\sigma\in G$.}
\end{equation}
}}
\end{definicion}

\begin{teorema}\label{finitud_TC/L}
Let $L/K$ be a finite Galois extension, with Galois group $G$. Consider the
function $\phi:\cogalois(L/K)\rightarrow Z^{1}(G,\mu(L))$, given by
$\phi(u + K) = f_{u}$ where $f_{u}(\sigma) = \sigma(u) - u$. Then
$\phi$ is an isomorphism of groups.
\end{teorema}

\begin{proof}
Let $\theta:T(L/K)\rightarrow Z^{1}(G,\mu(L))$ be given by
$\theta(u) = f_{u}$. Observe that $f_{u}(\sigma\circ\tau) =
\sigma(\tau(u)) - u$. Further $f_{u}(\sigma) = \sigma(u) - u$ and
$f_{u}(\tau) = \tau(u) - u$. If we apply $\sigma$ to this last
equality we obtain $\sigma(f_u(\tau)) = \sigma(\tau(u)) - \sigma(u)$.
The sum of this equation with the first one gives that $f_{u}$ is a
crossed homomorphism. We also note that if $\sigma\in G$ then
$f_{u}(\sigma) = \sigma(u) - u$ belongs to $\mu(L)$, since there is
an $N\in R_{T}$ such that $u^{N}\in K$. Therefore $(\sigma(u) -
u)^{N} = (\sigma(u))^{N} - u^{N} = \sigma(u^{N}) - u^{N} = 0$.

Furthermore, $\theta(u + v) = f_{u + v}$ and $f_{u + v}(\sigma) =
\sigma(u + v) - (u + v) = \sigma(u) + \sigma(v) - u - v = \sigma(u)
- u + \sigma(v) - v$, that is, $\theta(u + v) = \theta(u) +
\theta(v)$. Therefore $\theta$ is a homomorphism. If $u\in
\ker(\theta)$, then $\theta(u) = f_{u} = 0$. That is, $f_{u}(\sigma)
= \sigma(u) - u = 0$ and since $L/K$ is a Galois extension, it
follows that $u\in K$.

Conversely, if $u\in K$ we have $\theta(u) = 0$. Thus $\ker(\theta)
= K$ and therefore we obtain a monomorphism of abelian groups
\[
\phi:\cogalois(L/K)\rightarrow Z^{1}(G,\mu(L)).
\]

On the other hand, $Z^{1}(G,\mu(L))\subseteq Z^{1}(G,L)$ and from
the additive Hilbert Theorem 90, it follows that
\[
Z^{1}(G,L) = B^{1}(G,L)
= \{f\in Z^{1}(G,L)\mid \text{there exists $u\in L$ such that $f =
f_{u}$}\}.
\]

Hence, given $f\in Z^{1}(G,\mu(L))$ there exists $u\in L$ such that
$f = f_{u}$. Therefore for each $\sigma\in G$, $f(\sigma) =
f_{u}(\sigma) = \sigma(u) - u \in \mu(L)$. 
Let $K^{\prime}$ be the Galois closure of
$K(u)/K$. Then $K\subseteq K(u)\subseteq K^{\prime}\subseteq
L$.

Let $H = \Gal(L/K^{\prime})$. We have that $H\lhd G$ and that
$\textrm{card}(G/H)$ is finite. The conjugates of $u$ in
$K^{\prime}$ are $\{\overline{\sigma}(u)\mid \overline{\sigma}\in
\overline{G} = G/H\}$, thus
\[
\overline{\sigma}(u) = \sigma(u) = u + z_{\sigma}
\text{ where $z_{\sigma}\in \mu(L)$}.
\]

Since there is only a finite number of elements
$\overline{\sigma}\in \overline{G} = \{\sigma_{1}H,\ldots,
\sigma_{s}H\}$ there exist $N_{\sigma_{1}},\ldots, N_{\sigma_{s}}\in
R_{T}$ such that $(z_{\sigma_{i}})^{N_{\sigma_{i}}} = 0$. Let $N =
N_{\sigma_{1}}\cdots N_{\sigma_{s}}$. Then
\[
\overline{\sigma}(u^{N})= (u + z_{\sigma})^{N} 
= u^{N} + z_{\sigma}^{N} = u^{N} +
(z^{N_{\sigma}}_{\sigma})^{P_{\sigma}}= u^{N}.
\]

Since the extension $K^{\prime}/K$ is Galois we have that $u^{N}\in
K$, that is, $\phi$ is surjective.
\end{proof}

From Theorem \ref{finitud_TC/L} we obtain

\begin{proposicion}\label{dummit}
Let $E/F$ be a finite Galois extension, $\Gamma = \Gal(E/F)$ and let
$\Delta$ be a normal subgroup of $\Gamma$. Then the sequence of abelian
groups
\[
0\rightarrow Z^{1}(\Gamma/\Delta, \mu(E/F)^{\Delta})
\stackrel{\theta_{1}}{\rightarrow} Z^{1}(\Gamma,\mu(E/F))
\stackrel{\theta_{2}}{\rightarrow} Z^{1}(\Delta,\mu(E/F))
\]
is exact, where $\mu(E/F)^{\Delta}= \{\zeta\in \mu(E/F)\mid
\sigma(\zeta) = \zeta,\, \forall \sigma\in \Delta\}$.
\end{proposicion}

\begin{proof}
Suppose that $\theta_{1}(f) = 0$. Then if $\overline{\sigma}\in
\Gamma/\Delta$, we have $f(\overline{\sigma}) =
\theta_{1}(f)(\sigma) = 0$. Thus $\theta_{1}$ is injective. On the
other hand, ${\rm im}(\theta_{1})\subseteq \ker(\theta_{2})$ since
if $f = \theta_{1}(f^{\prime})$, with $f^{\prime}\in
Z^{1}(\Gamma/\Delta, \mu(E/F)^{\Delta})$, we have
$\theta_{2}(f)(\sigma) = \theta_{1}(f^{\prime})(\sigma) =
f^{\prime}(\overline{\sigma}) = 0$.

Now if $f\in \ker(\theta_{2})$ then for each $\sigma\in \Delta$, we
have $f(\sigma) = 0$. We may define
$f^{\prime}:\Gamma/\Delta\rightarrow \mu(E/F)$ by
$f^{\prime}(\overline{\sigma}) = f(\sigma)$. By the properties of
$f$, $f^{\prime}$ is well defined and it is a crossed homomorphism.
Finally, if $\tau\in \Delta$ then
$\tau(f^{\prime}(\overline{\sigma})) = \tau(f(\tau^{-1}\circ\sigma))
= \tau(f(\sigma)) = f^{\prime}(\overline{\sigma})$, that is,
$f^{\prime}\in Z^{1}(\Gamma/\Delta, \mu(E/F)^{\Delta})$ and $f =
\theta_{1}(f^{\prime})$.
\end{proof}

\begin{corolario}\label{cor_finitud_TC/L}
Let $L/K$ be a finite Galois extension. If the cardinality of
$\mu(L)$ is finite then the $R_{T}$-module $\cogalois(L/K)$ is
finite.
\end{corolario}
\begin{proof}
It follows from Theorem \ref{finitud_TC/L}.
\end{proof}

\section{Some structure results of radical
cyclotomic extensions}\label{teoremas_de_estructura}

\begin{proposicion}\label{cogalois_disjunto}
Let $L/K$ be a field extension such that $[L:K]=\ell$ with $\ell$ a
prime different from $p = \caracteristica(k)$. Then $L/K$ is not a
radical cyclotomic extension.
\end{proposicion}

\begin{proof}
Suppose that $L/K$ is radical cyclotomic. Then $\cogalois(L/K)$ is
non-trivial, since there is a subset $A\subseteq T(L/K)$ such that
$L=K(A)$. Let $\overline{\alpha}\in \cogalois(L/K)$ be different
from $0$, that is, $\alpha\notin K$. Let $M\in
R_{T}$ be such that $\alpha^{M}\in K$. We may assume that $M$ is a
monic polynomial and that it is the minimum degree polynomial with
such property, that is the order of $\overline{\alpha}$ is $M$.
Replacing $\alpha$ if necessary, we may assume that $M=Q$
is irreducible and $\alpha^{Q} = a\in K$.

Let $f(X) = {\rm irr}(\alpha,K)\in K[X]$. Since $\alpha^{Q}=a$, we
have $f(X)\mid X^{Q}- a$.
Therefore $f(X) = \prod(X - (\alpha +
\lambda^{B}_{Q}))$, for some polynomials $B\in R_{T}$. Observe that
$\grado(f(X))= \ell$, since $L=K(\alpha)$, and that $\sum(\alpha +
\lambda^{B}_{Q})= \ell\alpha + \lambda^{\sum B}_{Q}\in K$. On the
other hand, since $\ell\neq p$ we have $\ell\neq 0$ in $K$.
Therefore we have that $D = \sum B$ is non-zero, since
otherwise $\alpha\in K$ and the degree of $D$
is less than the degree of $Q$.

We have $\lambda^{D}_{Q}\notin K$. Since $\lambda^{D}_{Q}\in L$,
by purity, $\lambda^{D}_{Q}\in K$, which is a
contradiction. Therefore $L/K$ is not a radical cyclotomic
extension.
\end{proof}

\begin{corolario}\label{cogalois_disjuntocor1}
Let $L/K$ be a finite Galois extension such that $[L:K]= p^{s}n$, with
$p\nmid n$ and $n>1$, where $p=\caracteristica(K)$. Then $L/K$ is
not a radical cyclotomic extension.
\end{corolario}
\begin{proof}
From Cauchy's theorem, the group $G = \Gal(L/K)$ contains an element
of order $\ell$, say $g$, where $\ell$ is a prime that divides $n$.
Consider the subgroup $H = (g)$ of $G$. If $L/K$ were a radical
cyclotomic extension then, by Proposition \ref{prop_cog}, the
extension $L/L^{\prime}$, where $L^{\prime} = L^{H}$, would be a
radical cyclotomic extension. Since $[L:L^{\prime}] = \ell$ by
Proposition \ref{cogalois_disjunto} such extension is not a radical
cyclotomic extension. Therefore $L/K$ is not a radical cyclotomic
extension.
\end{proof}

\begin{corolario}\label{estructura1}
If $L/K$ is a finite Galois radical cyclotomic extension, then $[L:K]=
p^{s}$ with $p=\caracteristica(K)$, for some $s\in {\mathbb{N}}$.
\hfill $\square$
\end{corolario}

\begin{lema}\label{pureza_p}
Let $L/K$ be an extension such that $[L:K] = p^{s}$ with $s\in
{\mathbb{N}}$ and $p=\caracteristica(K)$. Then $L/K$ is pure.
\end{lema}
\begin{proof}
Assume that $L/K$ is not pure. Thus there exists $a=\lambda_{P}\in
L$ with $P\in R_{T}$ an irreducible polynomial such that $a^{P}= 0$ and
$a\notin K$. Consider the following diagram
\[
 \xymatrix{& L\ar@{-}[d]\\
 k(\lambda_{P})\ar@{-}[d]\ar@{-}[r]  & k(\lambda_{P})K = K(\lambda_{P}) \ar@{-}[d]\\
k \ar@{-}[r] & K}
\]

Let $\widetilde{K}=K\cap k(\lambda_{P})$. Then by Galois Theory we
have that $K(\lambda_{P})/K$ is a Galois extension with Galois group
$G$ isomorphic to $\Gal(k(\lambda_{P})/\widetilde{K})$. Thus
\[
\mid G\mid\mid [L:K]=p^{s}\,\, \text{and on the other hand} \mid G\mid\mid (q^{d}-1)
\]
where $d=\grado(P)$. Therefore $\mid G\mid = 1$, that is,
$\lambda_{P}\in K$.
\end{proof}

\begin{ejemplo}\label{schu_cogalois}
{\rm{A {\it Carlitz-Kummer} extension, see \cite{Schultheis}, is an
extension $L/K$ such that

(1) $K$ is a finite extension of $k(\Lambda_{M})$ for some $M\in
R_{T}$.

(2) $L$ is the splitting field of the polynomial
$f(u)=u^{M}-z\in K[u]$ over $K$, where $z\in K\setminus K^{M}$.

From Proposition 2.3 (4) of \cite{Schultheis}, it follows that
$[L:K]=p^{t}$, where $p=\caracteristica(K)$. Now Lemma
\ref{pureza_p} shows that Carlitz-Kummer extensions are radical
cyclotomic extensions.\hfill $\square$}}
\end{ejemplo}

Furthermore, from the above results, we obtain

\begin{teorema}\label{tdim_p_galois}
A finite Galois extension, $L/K$ is radical cyclotomic if and only if it is
radical, separable and $[L:K]=p^{s}$ with $s\in {\mathbb{N}}$ and
$p=\caracteristica(K)$. \hfill $\square$
\end{teorema}

We have the following theorem.

\begin{teorema}\label{tdim_p}
If $L/K$ is a finite radical cyclotomic extension, then $[L:K]=p^{n}$ for
some $n\geq 0$, where $p=\caracteristica(K)$.
\end{teorema}

\begin{proof}
Let $L/K$ be a finite radical cyclotomic extension, $L=K(
\alpha_1,\ldots,\alpha_t)$ for some $\alpha_i\in T(L/K)$. Then we have
\[
[L:K]=[L:K(\alpha_1,\ldots,\alpha_{t-1})]\cdots[K(\alpha_1,\alpha_2):
K(\alpha_1)][K(\alpha_1):K].
\]
Since each extension $K(\alpha_1,\ldots,\alpha_i)/K(\alpha_1,\ldots,
\alpha_{i-1})$ is finite radical cyclotomic, it suffices to handle
the case $L=K(\alpha)$.

Suppose that $L=K(\alpha)$ with $\alpha^M\in K$ for some $M\in
R_T$. Consider the factorization of $M$ in terms of irreducible polynomials
$M=P_1^{e_1}P_2^{e_2}\cdots P_s^{e_s}$. Let
$\beta_i:=\alpha^{M/P^{e_i}}$ for $1\leq i\leq s$. Then $L=
K(\beta_1,\ldots,\beta_s)$. Applying the same argument as before,
it suffices to deal with the case $L=K(\alpha)$ with $\alpha^{P^e}
\in K$ for some irreducible polynomial $P\in R_T$.

Let $L=K(\alpha)$ such that $\alpha^{P^e}\in K$ for some irreducible
polynomial $P\in R_T$. Let $\gamma_i:= \alpha^{P^{e-i}}$ for 
$1\leq i\leq e$, then $K(\gamma_1)\subseteq K(\gamma_2)\subseteq
\cdots\subseteq K(\gamma_e)=L$. So we have
\[
[L:K]=[L:K(\gamma_{e-1})]\cdots[K(\gamma_2):K(\gamma_1)][K(
\gamma_1):K].
\]
Since each extension $K(\gamma_i)/K(\gamma_{i-1})$ is finite
radical cyclotomic and $\gamma_i^P=\gamma_{i-1}$, it suffices
to consider the case $L=K(\alpha)$ with $\alpha^P\in K$ for
some irreducible polynomial $P\in R_T$.

Suppose that $\lambda_{P}\in L$. Then $L/K$ is a Galois
extension, since $L$ is the splitting field of the polynomial
$X^{P}-\alpha^P\in K[X]$ over $K$. From Corollary \ref{estructura1}, we
know that $L/K$ is a $p$-extension.

Now suppose that $\lambda_P\notin L$. We consider the diagram
\[
 \xymatrix{& & L=K(\alpha)\ar@{-}[r]^{a} & L(\lambda_{P})=K(\lambda_{P},\alpha)\\ 
     & K\ar@{-}[r]^{d} & K(\alpha)\cap K(\lambda_{P})\ar@{-}[r]^{a}\ar@{-}[u]^{b}&  K(\lambda_{P})\ar@{-}[u]^{b}\\ 
     k\ar@{-}[r]&K\cap k(\lambda_{P})\ar@{-}[r]^{d}\ar@{-}[u]^{c} & K(\alpha)\cap k(\lambda_{P})\ar@{-}[r]^{a}\ar@{-}[u]^{c}&
     k(\lambda_{P})\ar@{-}[u]^{c}}
\]

Since $K(\lambda_{P},\alpha)/K(\lambda_{P})$ is a Galois extension,
from Propositions 2.2 and 2.3 of \cite{Schultheis}, we have
$N=\Gal(L(\lambda_{P})/K(\lambda_{P}))$ can be considered a subgroup
of $\Lambda_{P}$, that is, $N$ is an elementary abelian $p$-group
and $\mid N\mid=b=p^{n}$.

Since
\[
[L:K]=[L:K(\alpha)\cap K(\lambda_{P})][K(\alpha)\cap K(\lambda_{P}):K]=bd=p^{n}d
\]
it suffices to show that $d=1$.

Let $H=\Gal(L(\lambda_{P})/(K(\alpha)\cap K(\lambda_{P})))$,
$G=\Gal(L(\lambda_{P})/K)$ and recall
$N=\Gal(L(\lambda_{P})/K(\lambda_{P}))$. Note that $N$ is a normal
subgroup of $G$.

We have
\[
G/N\cong \Gal(K(\lambda_{P})/K)<
\Gal(k(\lambda_{P})/k)\cong
C_{q^{d}-1}.
\]

Therefore $G/N$ is a cyclic group of order dividing $q^{d}-1$
and relatively prime to $p$. Furthermore, we have
$\mid G/N\mid=ad$.

From Hall's Theorem, see \cite{hall} Theorem 9.3.1, since $G$ is
a solvable group, there exists a cyclic subgroup $R$ of $G$, of order $ad$,
such that $G=NR$ (in fact, $G$ is the semidirect product
$G\cong N\rtimes R$ because $(\mid
R\mid,\mid N \mid)=1$).

Again from Hall's Theorem, any subgroup of order a divisor of $\mid
R\mid =ad$ is contained in a conjugate $R^{\prime}$ of $R$ and we
have $G=NR^{\prime}\cong N\rtimes R^{\prime}$.

Let $S=\Gal(L(\lambda_{P})/K(\alpha))\cong C_{a}$. Therefore we may
assume $S\subseteq R$ and $\mid R/S\mid=d$. Note that $(d,p)=1$.

Let $E=L(\lambda_{P})^{R}$. Observe that $L(\lambda_{P})^{S}=
K(\alpha)=L$. Therefore $K\subseteq E\subseteq L$,
$[L:E]=[R:S]=d=\mid R/S\mid$. Also, since $L/K$ is a radical
cyclotomic extension so it is $L/E$. Therefore $d=1$.
\end{proof}

\begin{corolario}\label{tdim_p_colateral}
With the notations of Theorem \ref{tdim_p} we have $K(\alpha)\cap
K(\lambda_{P})=K$ and $[L:K]=[L(\lambda_{P}):K(\lambda_{P})]$.
Furthermore
\[
{\rm Irr}(u,\alpha,K)={\rm
Irr}(u,\alpha,K(\lambda_{P}))=F_{1}(u)=\prod(u-(\alpha+\lambda^{A}_{P})).
\]
\end{corolario}
\begin{proof}
It follows from the proof of Theorem \ref{tdim_p}.
\end{proof}

\begin{corolario}\label{cogalois_grado_p_pureza}
A finite extension $L/K$ is radical cyclotomic if and only if it is
separable, radical and $[L:K]=p^{m}$ for some $m\in {\mathbb{N}}$.
\end{corolario}
\begin{proof}
It follows from Theorem \ref{tdim_p} and Lemma \ref{pureza_p}.
\end{proof}

\section{Examples and applications}\label{examples}

In this section we will present some applications of the above
results. First, we have the following consequence of Theorem
\ref{finitud_TC/L}.

\begin{proposicion}\label{finitud_redes}
If $E/L$ is a finite Galois field extension, with Galois group
$\Gamma$, then the function:
\[
\phi:\{H\mid L\leq H \leq T(E/L)\}\rightarrow \{U\mid U\leq Z^{1}(\Gamma,\mu(E))\},
\]
given by $\phi(H) = \{f_{\alpha}\in
Z^{1}(\Gamma,\mu(E))\mid \alpha\in H\}$, is a lattice isomorphism.
\end{proposicion}

\begin{proof}
It follows from the isomorphism given in Theorem \ref{finitud_TC/L}.
\end{proof}

Now, let $E/L$ be a Galois extension with Galois group $\Gamma$. We
define
\[
f:\Gal(E/L)\times \cogalois(E/L)\rightarrow \mu(E)
\]
by $f(\sigma,\overline{u}) = \sigma(u) - u$. Since
$\cogalois(E/L) \rightarrow Z^{1} (\Gamma, \mu(E)) $ is an
isomorphism, we have the evaluation function
\[
<,>:\Gamma\times Z^{1}(\Gamma, \mu(E))\rightarrow \mu(E)
\]
given by $<\sigma,h> = h(\sigma)$.

For each $\Delta \leq \Gamma $, $U \leq Z^{1} (\Gamma, \mu (E)) $
and each $\chi \in Z^{1} (\Gamma, \mu (E)) $ we define:
\begin{gather*}
\Delta^{\bot} = \{h\in Z^{1}(\Gamma, \mu(E))\mid <\sigma,h> = 0 
\text{ for every $\sigma\in \Delta$}\},\\
U^{\bot} = \{\sigma\in \Gamma\mid <\sigma,h> = 0 \text{ for every $h\in U$}\},\\
\chi^{\bot} = \{\sigma\in \Gamma\mid <\sigma,\chi> = 0\}.
\end{gather*}
Thus $\Delta^{\bot}\leq Z^{1}(\Gamma,\mu(E))$ and
$U^{\bot}\leq \Gamma$.

\begin{proposicion}\label{redes_casineat}
Let $E/L$ be a finite Galois extension with Galois group $\Gamma$.
Let $L^{\prime} /L$ be a subextension of $E/L$. Then $L^{\prime}/L$
is radical if and only if there exists a subgroup $U\leq
Z^{1}(\Gamma,\mu(E))$ such that $\Gal(E/L^{\prime}) = U^{\bot}$.
\end{proposicion}

\begin{proof}
If $L^{\prime}/L$ is a radical extension there exists
$\widetilde{G}\subseteq T(E/L)$ such that $L^{\prime} =
L(\widetilde{G})$. We may replace $\widetilde{G}$ by the {\it
additive} subgroup generated by $\widetilde{G}$ and $L$ which we
denote by $G$. Thus $L\leq G\leq T(E/L)$ and $L^{\prime} = L(G)$.
Let
\[
U = \phi(G)=\{f_{\alpha}\mid \alpha\in G\}\leq Z^{1}(\Gamma,\mu(E))
\]
where $\phi$ is the function given in Proposition
\ref{finitud_redes}. Then
\begin{align*}
U^{\bot} &= \{\sigma\in \Gamma\mid <\sigma,f_{\alpha}> = 0 \text{ for each $f_{\alpha}\in U$}\}\\
&=\{\sigma\in \Gamma\mid f_{\alpha}(\sigma) = 0\text{ for each $f_{\alpha}\in U$}\}\\
&=\{\sigma\in \Gamma\mid  \sigma(\alpha)=\alpha \text{ for each $f_{\alpha}\in U$}\}\\
&=\{\sigma\in \Gamma\mid \sigma(x) = x \text{ for each $x\in L(G)$}\}\\
&=\Gal(E/L(G))) = \Gal(E/L^{\prime}).
\end{align*}

Conversely, if there exists a subgroup $U \leq Z^{1} (\Gamma, \mu(E))
$ such that $\Gal(E/L^{\prime}) = {U^\bot}$ then we will see that
$\Gal(E/L^{\prime}) = U^{\bot} = \Gal(E/L(G)))$, with $G=\{\alpha\in
T(E/L)\mid f_{\alpha}\in U\} = \phi^{-1}(U)$ where $\phi$ is the function
of Proposition \ref{finitud_redes}.

To prove the above equalities we only have to show that $U^{\bot} =
\Gal(E/L(G))$. For this purpose we consider $\tau\in U^{\bot} =
\{\sigma\in\Gamma\mid h(\sigma)=0\, \forall\, h\in U\}$. If
$\alpha\in G$ then $f_{\alpha}\in U$. In particular,
$f_{\alpha}(\tau) = 0 = \tau(\alpha) - \alpha$. Therefore for each
$\alpha\in G$, we obtain $\tau(\alpha) = \alpha$ and thus $\tau$
fixes $L(G)$. Thus $\tau\in\Gal(E/L(G))$. Now let $\tau\in
\Gal(E/L(G))$ and $h\in U$. By the definition of $G$ there exists
$\alpha \in G$ such that $h = f_{\alpha}$. Therefore $\phi$ is
bijective. It follows that $h(\tau) = f_{\alpha}(\tau) = 0$. Hence
$\tau \in U^{\bot}$. From Galois theory, it follows that $L^{\prime}
= L(G)$.
\end{proof}

The following result is an application of Proposition
\ref{redes_casineat}, see \cite{barrera2}. The symbol
$\sqrt[N]{\alpha} $ denotes a root of the polynomial $u^{N} -
\alpha$.
\begin{proposicion}\label{redes2}
Let $K/F$ be a finite separable extension and let $E$ be the normal
closure of $K/F$. Suppose that there is a finite extension $L/F$
such that

{\rm (1)} $E(\lambda_{N})\cap L = F$ where $N\in R_{T}$ is a nonconstant
polynomial.

{\rm (2)} $KL = L(\sqrt[N]{\alpha})$ for some nonzero $\alpha\in L$.

Then $K=F(\sqrt[N]{\alpha})$.
\end{proposicion}

\begin{proof}
Consider the following diagram
\[
 \xymatrix{E(\lambda_{N})\ar@{-}[d]\ar@{-}[r]& E(\lambda_{N})L\ar@{-}[d]\\
 K\ar@{-}[d]\ar@{-}[r]  & KL \ar@{-}[d]\\
F \ar@{-}[r] & L}
\]

Since the extension $E(\lambda_{N})/F$ is a Galois extension, we
have that $E(\lambda_{N})L/L $ is a Galois extension and from (1) we
obtain
\[
G=\Gal(E(\lambda_{N})/F)\cong \Gal(E(\lambda_{N})L/L)=G_{1}.
\]

From (2) we have $KL = L(\sqrt[N]{\alpha}) $. Let $\beta =
\sqrt[N]{\alpha}$ and $\sigma\in \Gal(E(\lambda_{N})L/L)$.
Then
\begin{equation}\label{ec4}
  (\sigma(\beta)- \beta)^{N} = \sigma(\beta^{N})-\beta^{N} = \sigma(\alpha)-\alpha = 0.
\end{equation}

We define $\chi: G_{1} \rightarrow \mu(E (\lambda_{N}) L) $ by
$\chi(\sigma) = \sigma(\beta) - \beta$.

If $\sigma \in \Gal(E(\lambda_{N})L/KL) $ we have $\chi(\sigma) =
\sigma(\beta) - \beta = 0$ and vice versa. Thus
$\Gal(E(\lambda_{N})L/KL) = \ker(\chi)$. Further, from (\ref{ec4})
we have that the image of $\chi$ is contained in $\Lambda_{N}$.
Since $G$ and $G_{1}$ are isomorphic, $\chi$ can be defined on $G$.

Therefore $\chi$ can be considered as an element of $Z^{1}(G,
E(\lambda_ {N})) $ and $\ker(\chi) $ as equal to
$\Gal(E(\lambda_{N})/K)$. From Proposition \ref{redes_casineat}
$K/F$ is a radical extension.
\end{proof}

Let $E/F$ be a finite Galois extension with Galois group $G$. Let
$L/F$ be another extension such that $L\cap E = F$. Consider the
composition $EL$. The restriction
\[
\Gal(EL/L)\rightarrow \Gal(E/F),\,\,\, \sigma\mapsto \sigma\mid_{E}
\]
is an isomorphism of groups. Denote by $S(L_{1}/L_{2})$
the set of extensions of $L_{2}$ contained in $L_{1}$. Then the
functions
\begin{gather*}
\varepsilon:S(E/F)\rightarrow S(EL/L),\,\,\, K^{\prime}/F\mapsto LK^{\prime}/L\\
\intertext{and}
\lambda:S(EL/L)\rightarrow S(E/F),\,\,\, K_{1}/L\mapsto (K_{1}\cap E)/F
\end{gather*}
are lattice isomorphisms, inverses of each other.

Denote by $ST(E/F)$ the set of all radical sub-extensions
$K^{\prime}/F$ of $E/F$. Then for $K^{\prime}/F\in ST(E/F)$ there
exists an $R_{T}$-module $G$, which is not necessarily unique, such
that $F\subseteq G\subseteq T(E/F)$ and $K^{\prime} = F(G)$. Let
$G_{1} = G + L$. Since $LK^{\prime} = L(G)$, $L\subseteq
G_{1}\subseteq T(EL/L)$ and $G_{1}$ is an $R_{T}$-module, we have
$LK^{\prime}= L(G_{1})$. Therefore $\varepsilon(K^{\prime}/L)\in
ST(EL/L)$. Thus the restriction of $\varepsilon$ to radical extensions
defines an injective function
\begin{gather*}
\rho:ST(E/F)\rightarrow ST(EL/L)\\
\intertext{given by}
F(G)/F\mapsto F(G)L/L = L(G + L)/L
\end{gather*}
where $G$ is an $R_{T}$-module such that $F\subseteq G
\subseteq T(E/F)$.

\begin{proposicion}\label{redes_torsion_galois}
Let $E/F$ be a finite Galois extension with Galois group $\Gamma$,
and let $L/F$ be an arbitrary extension, with $L\subseteq
\overline{k}$ and such that $E\cap L = F$. If $\mu(EL) = \mu(E)$,
then:
\begin{description}
  \item[(1)] $(G + L) \cap E = G$ for every $R_{T}$-module $G$ with
  $F\subseteq G\subseteq T(E/F)$
  \item[(2)] $G_{1} = (G_{1}\cap E) + L$ for every $R_{T}$-module
  $G_{1}$ with $L\subseteq G_{1} \subseteq T(EL/L)$.
  \item[(3)] The function
\begin{gather*}
\rho:ST(E/F)\rightarrow ST(EL/L)\\
F(G)/F\mapsto L(G +L)/L, \,\,\, F\leq G\leq T(E/F)\\
\intertext{is bijective, and the function}
ST(EL/L)\rightarrow ST(E/F),\\
L(G_{1})/L\mapsto F(G_{1}\cap E)/F,\quad L\leq G_{1} \leq T(EL/L)
\end{gather*}
is the inverse of $\rho$.
\end{description}

Here, the notation $F \leq G$ means that $F$ is a
submodule of the $R_{T}$-module $G$.
\end{proposicion}

\begin{proof}
(1) Let $w\in (G + L)\cap E$. Thus $w = x + y$ where $x\in G$ and
$y\in L$. Therefore $y = w - x\in E$. Hence $y\in F$ since $E\cap L
= F$. Therefore $w\in G$. Conversely if $x \in G$ certainly $x \in
(G + L)\cap E$.

(2)Denote by $\Gamma_{1}$ the Galois group of $EL/L$. We have seen
that there exists a group isomorphism
\begin{equation}\label{ec3}
\theta:\Gamma_{1}\rightarrow \Gamma,\,\,\, \sigma_{1}\rightarrow
\sigma_{1}\!\!\mid_{E}.
\end{equation}

Since $\mu(EL) = \mu(E)$, $\theta$ induces a group isomorphism
\[
\upsilon:Z^{1}(\Gamma,\mu(E))\rightarrow Z^{1}(\Gamma_{1},\mu(EL))
\]
given as follows: let $h\in Z^{1}(\Gamma,\mu(E))$. If
$\sigma_{1}\in \Gamma_{1}$ then $\sigma_{1}\!\!\mid_{E}\in \Gamma$.
Define $\upsilon(h)(\sigma_{1}) := h(\sigma_{1}\mid_{E})$. We have
\[
\upsilon(h)(\sigma_{1}\circ
\sigma_{2}) = h(\sigma_{1}\circ \sigma_{2}\mid_{E}) =
h(\sigma_{1}\mid_{E}\circ \sigma_{2}\mid_{E}).
\]

Thus $\upsilon(h)$ is a crossed homomorphism. By construction
$\upsilon$ is a group homomorphism, and by (\ref{ec3}) we have
$\upsilon$ is a group isomorphism.

Let $G_{1}$ be such that $L\leq G_{1} \leq T(EL/L)$. If $w\in
(G_{1}\cap E) + L$ then $w = x + y$ with $x\in (G_{1}\cap E)$ and
$y\in L$. Thus $w\in G_{1}$. Now let $a_{1}\in G_{1}$. Then
$f_{a_{1}}\in Z^{1}(\Gamma_{1},\mu(EL))$. There exists $f \in Z^{1}
(\Gamma, \mu(E)) $ such that $f_{a_{1}} = \upsilon(f)$. From
Proposition \ref{finitud_TC/L}, we have that there exists $a \in
T(E/F) $ such that $f_{a_{1}} = \upsilon(f = f_{a})$.

We have $f_{a_{1}}(\sigma_{1}) = f_{a}(\sigma_{1}\mid_{E})$ for all
$\sigma_{1}\in \Gamma_{1}$. It follows that $\sigma_{1} (a_{1}) -
a_{1} = \sigma_{1} (a) - a$, that is, $\sigma_{1}(a_{1} - a) = a_{1}
- a$ for each $\sigma_{1}\in \Gamma_{1}$. Thus $a_{1} - a\in L$.
Therefore $a_{1} = a + b$ where $b\in L$. Since $a\in G_{1}$ it
follows that $a\in (G_{1}\cap E) + L$.

(3) From the remark before this proposition we have that $\rho$ is
injective, so it suffices to show that $\rho$ is surjective. Let
$K_{1}/L\in ST(EL/L)$. Then $K_{1} = L(G_{1})$ for some $G_{1}$ with
$L\leq G_{1}\leq T(EL/L)$. Therefore if we put $G = G_{1}\cap E$, we
obtain that $F(G)/F\in ST(E/F)$ and that
\[
\rho(F(G)/F) = L(F(G))/L = L(F(G_{1}\cap E))/L = L(L + (G_{1} \cap E))/L.
\]
From (2), we have $L(L + (G_{1} \cap E))= L(G_{1}) =
K_{1}$.
\end{proof}

The following lemma shows that the converse of Theorem \ref{tdim_p}
is not always valid.
\begin{lema}\label{auxiliar_contra_cogalois}
Let $L/K$ be a Galois extension such that $[L:K] = p^{2} $,
$\mu(L)=\mu(K)$ and $G = \Gal(L/K)\cong C_{p^{2}}$. Then $L/K$ is
not a radical extension.
\end{lema}
\begin{proof}
Suppose that $L/K$ is a radical extension. Consider the group
$H^{1}(G,\mu(L))$. Since $\mu(L)=\mu(K)$ we have
$B^{1}(G,\mu(L))=\{1\}$. Therefore,
$H^{1}(G,\mu(L))=Z^{1}(G,\mu(L))/B^{1}(G,\mu(L))\cong {\rm
Hom}(G,\mu(L))$. From Proposition \ref{finitud_TC/L}, we obtain that
\[
\cogalois(L/K)\cong {\rm Hom}(G,\mu(K)).
\]
By Cauchy's theorem, there is an element of order $p$, say $\tau$,
in $G$. Let $H=(\tau)$ and $L^{\prime}= L^{H}$. $L^{\prime}/K$ is a
Galois extension and $G^{\prime} = \Gal(L^{\prime}/K) $ is
isomorphic to $C_{p}$. Note that $\mu(L^{\prime})=\mu(K)$. Thus
\[
\cogalois(L^{\prime}/K)\cong {\rm Hom}(G^{\prime},\mu(K)).
\]
We will see that the cardinality of $\cogalois(L/K) \cong {\rm Hom}
(G, \mu(K)) $ is $\mid\!\! \mu(K) \!\!\mid$. Let $a \in G \cong
C_{p^{2}} $ be a generator. A homomorphism $\psi: G \rightarrow
\mu(K)$ is completely determined by its action on $a$. Hence there
exists $\mid\!\!\mu(K)\!\!\mid$ homomorphisms from $G$ to $\mu(K)$.
Similarly we can show that the cardinality of
$\cogalois(L^{\prime}/K) \cong {\rm Hom} (G^{\prime}, \mu(K)) $ is
$\mid\!\!\mu(K)\!\!\mid$.

On the other hand we have $\cogalois(L^{\prime}/K) \subseteq
\cogalois(L/K) $, see Proposition \ref{prop_cog}, and since both
have the same order, we have $\cogalois(L^{\prime}/K) =
\cogalois(L/K)$. It follows that $L = L^{\prime}$ since if
$\alpha_{1}, \ldots, \alpha_{s}$ generate $L$ over $K$, then
$\alpha_{1}, \ldots , \alpha_{s}\in L^{\prime}$. Therefore $
p^{2}=[L:K] = [L^{\prime}: K] = p$ which is a contradiction.
\end{proof}

The following example shows that the property of being radical
extension is not hereditary.

\begin{ejemplo}\label{noredes_cogalois}
{\rm{Let $M=P^{n}$, $n\in {\mathbb{N}}$ with $P\in R_{T}$
irreducible. Consider the field extension $k(\Lambda_{M})/
k(\Lambda_{P})$. Let $t \in {\mathbb{N}} $ be such that $p^{t-1} < n
\leq p^{t} $ and $n_{0}$ the integer part of $\frac{n}{p^{t-1}}$.

From Corollary 1 of \cite{Lam}, it follows that
\[
H_{M}\cong ({\mathbb{Z}}/p^{t}{\mathbb{Z}})^{\alpha}\times 
{\mathbb{Z}}/p^{n_{1}}{\mathbb{Z}}\times \cdots \times 
{\mathbb{Z}}/p^{n_{s}}{\mathbb{Z}}
\]
with $t> n_{1}\geq \cdots \geq n_{s}\geq 0$. Here $H_{M}$
is Galois group of the field extension
$k(\Lambda_{M})/k(\Lambda_{P})$.

Let $n=5$ and $p=3$. Then $p^{t-1} < n \leq p^{t}$ holds for $t=2$
and we have $n_{0} = 1$. The value of $\alpha$ is given in Corollary
1 of \cite{Lam}.

Choose $H$ a subgroup of $H_{M}$ of the form
\[
H=({\mathbb{Z}}/p^{t}{\mathbb{Z}})^{\alpha-1}\times {\mathbb{Z}}/p^{n_{1}}
{\mathbb{Z}}\times \cdots \times {\mathbb{Z}}/p^{n_{s}}{\mathbb{Z}}.
\]
Let $L^{\prime}=L^{H}$. Therefore
$\Gal(L^{\prime}/k(\Lambda_{P}))\cong C_{p^{2}}$. We have $\mu
(k(\Lambda_{P})) = \mu(L^{\prime})$ for suitable $q = p^{\nu}$.

From Lemma \ref{auxiliar_contra_cogalois} we have that $L^{\prime}
/k(\Lambda_{P})$ is not a radical extension. Therefore
$k(\Lambda_{P^{5}}) / k(\Lambda_{P}) $ is a Galois radical
cyclotomic extension, that does not satisfy the property that if $L$
is a field such that $ k(\Lambda_{P}) \subseteq L \subseteq
k(\Lambda_{P^{5}}) $, then $L/k(\Lambda_{P})$ is a radical
extension.\hfill $\square$}}
\end{ejemplo}

\begin{ejemplo}\label{ejemplo_schultheis}
{\rm{Let $q = p \geq 3 $. Consider the extension $L/k(\Lambda_{T})
$, where $L$ is the splitting field of the polynomial $f(X) =
X^{T} - 1$, with coefficients in $k(\Lambda_{T})$. The degree of
this extension is $[L: k(\Lambda_{T})] = p$, see Example
\ref{ejemplo4}. We study the structure of $\cogalois(L /
k(\Lambda_{T}))$.

Suppose that $\overline{\beta} \in \cogalois(L / k(\Lambda_{T})) $
has order $Q^{r}$, with $Q$ monic irreducible polynomial, $r\geq 1$
and $Q\neq T$. Since
$\lambda_{Q}=\lambda^{Q^{r-1}}_{Q^{r}}\in L$ and since
$L/k(\Lambda_{T})$ is a pure extension it follows that $\lambda_{Q}
\in k(\Lambda_{T}) $. From Proposition \ref{finitud_RC} we obtain
that $Q =T$ which is a
contradiction. Therefore
\[
\cogalois(L/k(\Lambda_{T}))\cong \cogalois(L/k(\Lambda_{T}))_{T}
\]
where $\cogalois(L/k(\Lambda_{T}))_{T}$ is the set of
elements of $\cogalois(L/ k(\Lambda_{T})) $ whose order is a power
of $T$.

Next we compute the cardinality of $\cogalois(L/k(\Lambda_{T}))$.
Here we need an auxiliary result.

Let $z \in k$, $z \neq 0$, and $N \in R_{T}$ be a nonconstant
polynomial. Consider $g(X) = X^{N} - z \in k(\Lambda_ {N})[X]$. The
splitting field of $g(X)$ over $k$ is of the form $K = k(\alpha,
\lambda_{N}) $ where $\alpha$ is an arbitrary root of $g(X)$ and
$\lambda_{N} $ is a generator of $\Lambda_{N}$. Since the polynomial
$g(X)$ is separable, $K/k$ is a Galois extension.

Let $G = \Gal(K/ k)$. Given $\sigma \in G$ we have that
$\sigma(\alpha) = \alpha + \lambda^{M_{\sigma}}_N$ and
$\sigma(\lambda_N) = \lambda^{N_{\sigma}}_N$, where $M_{\sigma}$ and
$N_{\sigma}$ are determined up to a multiple of $N$, and
$N_{\sigma}$ is relatively prime to $N$.

Let $G(N)$ be the subgroup of $GL_{2} (R_{T} / (N))$ of all matrices
of the form
\begin{displaymath}
\left(\begin{array}{ccc} 1 & 0  \\
 \overline{B} & \overline{A}
\end{array}\right),
\end{displaymath}
where $\overline{B}\in R_{T}/(N)$ and $\overline{A}\in
(R_{T}/(N))^{*}$. We have that ${\rm card} (G (N)) = q^{\grado(N)}
\Phi(N)$. Let $\theta: G \rightarrow G(N)$ be defined as follows:
\begin{displaymath}
\theta(\sigma) = \left(\begin{array}{ccc} 1 & 0  \\
 \overline{M_{\sigma}} & \overline{N_{\sigma}}
\end{array}\right).
\end{displaymath}

We have the following lemma.

\begin{lema}\label{auxiliar_noabeliano_galois}
Let $K/k$ and $\theta$ be as above. Then $\theta$ is a group
monomorphism. On the other hand, if $N = P$, where $P$ is a monic
irreducible polynomial and $z \in R_{T}$ is as before and the
equation $g(X) = 0$ has no solutions in $R_{T}$, then $\theta$ is a
group isomorphism.
\end{lema}

\begin{proof}
Let $\sigma,\tau\in G$. Then $\sigma(\tau(\alpha)) = \sigma(\alpha +
\lambda^{M_{\tau}}) = \alpha + \lambda^{M_{\sigma}} +
\lambda^{M_{\tau}N_{\sigma}}$ and $\sigma(\tau(\lambda)) =
\sigma(\lambda^{N_{\tau}}) = \lambda^{N_{\sigma} N_{\tau}}$.
Then
\begin{displaymath}
\theta(\sigma\cdot\tau) = \left(\begin{array}{ccc} 1 & 0  \\
 \overline{M_{\sigma} + M_{\tau}N_{\sigma}} & \overline{N_{\sigma}
 N_{\tau}}
\end{array}\right) = \left(\begin{array}{ccc} 1 & 0  \\
 \overline{M_{\sigma}} & \overline{N_{\sigma}}
\end{array}\right)\left(\begin{array}{ccc} 1 & 0  \\
 \overline{M_{\tau}} & \overline{N_{\tau}}
\end{array}\right)= \theta(\sigma)\theta(\tau)
\end{displaymath}
Therefore $\theta$ is a group homomorphism. If
$\theta(\sigma)$ is the identity matrix we have that $M_{\sigma}$ is
a multiple of $N$ and that $N_{\sigma} = 1 + NQ$. Thus $\theta$ is a
group monomorphism.

When $N = P$, where $P$ is a monic irreducible polynomial, $z \in
R_{T}$ is as before and the equation $g(X) = 0$ has no solutions in
$R_{T}$, then by Theorem 1.7 (4) of \cite{hsu}, we have that $\Gal(K
/ k(\lambda_{P}))$ is of order $q^{\grado(P)} $. Hence $\theta$ is
an isomorphism.
\end{proof}

Coming back to the example,
it will be shown that $\mu(L) = \Lambda_{T}$. To begin with,
certainly $\Lambda_{T} = \mu(k(\Lambda_{T})) \subseteq \mu(L)$. On
the other hand let $u \in \mu(L) $ be nonzero. There exists $N \in
R_{T}$ such that $u^{N} = 0$. Therefore $u$ is of the form
$\lambda^{M}_{N}$. We may assume that $(M, N) = 1$. Thus, from
Proposition 12.2.21 of \cite{villa_salvador}, we obtain that
$\lambda_{N} \in L$. Let $N=P^{\alpha_{1}}_{1}\cdots
P^{\alpha_{s}}_{s}$. Then
$\lambda_{P_{i}}=\lambda^{P^{\alpha_{1}}_{1}\cdots
P^{\alpha_{i}-1}_{i}\cdots P^{\alpha_{s}}_{s}}_{N}\in L$. Since
$L/k(\Lambda_{T})$ is pure it follows that $\lambda_{P_{i}} \in
k(\Lambda_{T})$. Thus $P_{i}=T$. Therefore $N=T^{n}$ with $n\in
{\mathbb{N}}$.

Suppose that $n \geq 2$ and consider the diagram
\[
 \xymatrix{& L\ar@{-}[dl]\ar@{-}[dr]  &\\
k(u)\ar@{-}[dr] & & k(\Lambda_{T})\ar@{-}[dl] \\
 & k&}
\]

We have $[L: k(u)] \Phi(T^{n}) = p (p-1)$. Since $\Phi(T^{n}) =
p^{n-1} (p-1) $, it follows that $[L: k(u)] p^{n-1} = p$. If $n\geq
3$ then $n-2\geq 1$. Thus $[L:k(u)]p^{n-2}=1$ which is a
contradiction. It only remains to consider the case $n = 2$, so that
$L = k (u)$. From Lema \ref{auxiliar_noabeliano_galois} we have that
$\Gal(L / k)$ is not an abelian group, contrary to the fact that the
group $\Gal(k(\Lambda_{T^2})/ k)$ is an abelian group. Therefore
$n=1$ and $u=\lambda^{M}_{T}\in k(\Lambda_{T})$.

From Lemma \ref{auxiliar_coho_cogalois} we have
that $B^{1} (G, \mu(L)) = \{0 \}$. 
Then $H^{1}(G,\mu(L))=Z^{1}(G,\mu(L))/B^{1}(G,\mu(L))\cong {\rm
Hom}(G,\mu(L))$. 
Thus, using the proof of Lemma
\ref{auxiliar_contra_cogalois}, we obtain $\mid \cogalois(L /
k(\Lambda_ {T})) \mid = [L: k(\Lambda_{T})] = p $.\hfill $\square$}}
\end{ejemplo}

\begin{ejemplo}\label{ejemplo_entre_ciclotomicos}
{\rm{Consider the extension $k(\Lambda_{P^n}) / k(\Lambda_{P})$.
We compute the order of $\cogalois (k (\Lambda_{P^n}) /
k(\Lambda_{P})) $ in the following case: $P = T$, $q = p > 2$ and $n
= 2$. Let $H_{T^{2}}=\{\overline{N}\in R_{T}/(T^{2})\mid (N,T^{2})=1
\text{ and } N\equiv 1 \bmod T\}$, then ${\rm card}(H_{T^{2}}) =
q^{d (n-1)} = p$, with $d = \grado(P (T)) = 1$. In particular
$H_{T^2}$ is a cyclic group. We have
\begin{gather*}
H^{1}(H_{T^{2}},\Lambda_{T^{2}}) \cong \ker(N_{H_{T^{2}}})/ D\Lambda_{T^{2}},\\
\intertext{where we define $N_{H_{T^{2}}}:\Lambda_{T^{2}}\rightarrow
\Lambda_{T^{2}}$ and $D:\Lambda_{T^{2}}\rightarrow \Lambda_{T^{2}}$
by}
N_{H_{T^{2}}}(x) = x + \sigma\cdot x + \cdots + \sigma^{p-1}\cdot x,\\
D(x) = \sigma\cdot x - x,
\end{gather*}
where $\sigma = 1 + T + (T^{2})$ is a generator of $H_{T^2}$ and
$x \in \Lambda_{T^2}$. On the other hand if $x =
\lambda^{M}_{T^{2}}$ we have
\[
N_{H_{T^{2}}}(x) = \lambda^{M}_{T^{2}}+ \lambda^{M(1+T)}_{T^{2}}
+\cdots +\lambda^{M(1+(p-1)T)}_{T^{2}} =\lambda^{pM+(1+2+\cdots +
p-1)MT}_{T^{2}} = 0.
\]
Note that $1 + 2 + \cdots + (p-1) = 0$ since $ \frac{p (p-1)}{2} =0$
in $ {\mathbb{F}}_{p}$. Thus we have that $\ker(N_{H_{T^2}}) =
\Lambda_{T^{2}} $. We also have $ D(x)
=\lambda^{M(1+T)}_{T^{2}}-\lambda^{M}_{T^{2}} = \lambda^{M}_{T}$.
Thus
\[
D \Lambda_{T^{2}} = \Lambda_{T}
\]
Therefore $H^{1}(H_{T^{2}},\Lambda_{T^{2}}) =
\Lambda_{T^{2}}/\Lambda_{T}$. On the other hand, from Lemma
\ref{auxiliar_coho_cogalois}, we have that ${\rm card} (B^{1}
(H_{T^2}, \Lambda_{T^{2}})) = {\rm card} (\Lambda_{T^{2}} /
\Lambda_{T}) $ and since
\[
H^{1}(H_{T^{2}},\Lambda_{T^{2}}) = Z^{1}(H_{T^{2}},
\Lambda_{T^{2}})/B^{1}(H_{T^{2}},\Lambda_{T^{2}})
\]
from Proposition \ref{finitud_TC/L} follows that
\[
\mid({\rm cog}(k(\Lambda_{T^{2}})/k(\Lambda_{T})))\mid =
 \mid(Z^{1}(H_{T^{2}},\Lambda_{T^{2}}))\mid =
[k(\Lambda_{T^{2}}):k(\Lambda_{T})]^{2}.
\]
}}
\end{ejemplo}

The following lemma shows that certain extensions have properties
analogous to those provided in Lemma 1.3 of \cite{greither}, namely
steps 1 and 2. However we will see that these properties do not hold
in general.

\begin{lema}\label{clave_cogalois}
Consider the extension $L/k(\lambda_{P}) $, where $L$ is the
splitting field of the polynomial $ X^{P}-a $, where $P \in R_{T}$
is irreducible and $a \in k(\lambda_{P}) \setminus
k(\lambda_{P})^{P}$. The module $\cogalois(L / k(\lambda_{P}))$ has
no elements of order $Q$, where $Q$ is an irreducible polynomial
different from $P$. Furthermore if $\nu_{{\mathfrak {p}}}(a) \geq
q^{d}$, where $d = \grado(P)$, we have that $\cogalois(L / K)$ has
no elements of order $P^{2}$.
\end{lema}

\begin{proof}
Suppose that there exists an irreducible element of order $Q$ in
$\cogalois(L/k(\lambda_{P}))$. Since $L / k(\lambda_{P}) $ is a
Galois extension, we have $\lambda_{Q} \in L$ and since $L /
k(\lambda_{P})$ is radical cyclotomic extension, from Proposition
\ref{finitud_RC} we have $\lambda_{Q} \in \mu(k(\lambda_P)) = \Lambda_{P}$.
Therefore $Q = P$.

Now assume that $\cogalois (L / k(\lambda_{P}))$ has an element of
order $P^{2}$, that is, there is $\bar{\beta} \in {\rm cog} (L /
K)$ such that $\beta^{P^{2}} = b \in k(\lambda_P)$. Then, since $L /
k(\lambda_{P})$ is radical, from Proposition
\ref{galois_torsion_raices} it follows that $\lambda_{P^2} \in L$.
Consider the following diagram
\[
 \xymatrix{{\mathcal{O}}_{L}\ar@{-}[d]\ar@{-}[r] & L\ar@{-}[d] \\
 {\mathcal{O}}_{k(\lambda_{P^{2}})}\ar@{-}[d]\ar@{-}[r]& 
k(\lambda_{P^{2}})\ar@{-}[d]\\
 {\mathcal{O}}_{k(\lambda_P)}\ar@{-}[d]\ar@{-}[r]& k(\lambda_P)\ar@{-}[d]\\
 R_{T}\ar@{-}[r] & k}
\]

Now the ramification index of the prime $P$ in the extension
$k(\lambda_{P}^{2})/k$ is $\Phi(P^{2}) $ so the ramification index
of $P$ in the extension $L /k$ is $\widetilde{d} \Phi(P^{2})$,
where $ \widetilde{d} = e_{L /k(\lambda_{P^2})} $. From Theorem
3.9 of \cite{Schultheis} we have that this ramification index is
$\Phi(P)$. In other words, $d \Phi(P^{2}) = \Phi(P)$, which is
absurd.
\end{proof}

\begin{ejemplo}\label{ejemplo6_1}
{\rm{Let $P, Q \in R_{T}$, be different irreducible polynomials.
Consider the extension $L = k(\Lambda_{P^{2} Q^{2}}) / k$. Note that
$ L = k(\lambda_{P^{2}}, \lambda_{Q^{2}}) $. Let $\sigma = 1 + PQ
\in G = \Gal(L / k)$. Since
$\lambda^{1+PQ}_{P^{2}Q^{2}}=\lambda_{P^{2}Q^{2}}+\lambda_{PQ}\neq
\lambda_{P^{2}Q^{2}}$, we have that $\sigma\neq 1$.

We have $\sigma(\lambda_{PQ})=\lambda^{1+PQ}_{PQ}=\lambda_{PQ}$.
Therefore if $K$ is the field fixed by $(\sigma)$, we have
$\lambda_{PQ} \in K $.

On the other hand $\sigma^{p} = (1 + PQ)^{p} = 1 + P^{p} Q^{p}
\equiv 1 \bmod \, P^{2} Q^{2} $, so that the order of $\sigma$ is
$p$. Hence $[L: K] = p$.

Since $\sigma(\lambda_{P^{2}}) = \lambda_{P^{2}} + \lambda^{Q}_{P}
\neq \lambda_{P^{2}} $ we have $\alpha=\lambda_{P^{2}}\notin K$.
Analogously it can be shown that $\beta = \lambda_{Q^{2}} \notin K$.

Since $[L: K] = p$, we have that $L = K (\alpha) = K (\beta)$,
$\alpha^{P} = \lambda_{P} $ and $\beta^{Q} = \lambda_{Q} $.
Therefore $\cogalois(L / K)$, contains elements of order $P$ and
elements of order $Q$.}}
\end{ejemplo}

\begin{ejemplo}\label{ejemplo7_1}
{\rm{Let $q=p^{\nu}$ with $p\geq 3$. Let $L = k(\Lambda_{P^{3}}) $ and
$\sigma = 1 + P \in \Gal(L/k(\Lambda_{P}))$. We have $\sigma^{p} =
(1 + P)^{p} \equiv 1 \bmod \, P^{3}$. Furthermore $\sigma \neq 1$
since $\sigma (\lambda_{P^{3}}) = \lambda_{P^{3}} + \lambda_{P^{2}}
\neq \lambda_{P^{3}}$.

Let $K = L^{(\sigma)}$. We have $[L: K] = p$. On the other hand
$\sigma(\lambda_{P^{2}}) = \lambda_{P^{2}} + \lambda_{P} \neq
\lambda_{P^{2}}$. Therefore $\alpha = \lambda_{P^{2}} \notin K$.
Thus $L = K(\alpha)$ and $\alpha^{P} = a \in K$.

Since $\lambda^{P^{2}}_{P^{3}}\in K$ and $\lambda^{P}_{P^{3}} \notin
K$, $\lambda_{P^{3}} \in L$ has order $P^{2} $. Therefore
${\cogalois} (L / K)$ has elements of order $P^{2}$.\hfill
$\square$}}
\end{ejemplo}

Examples \ref{ejemplo6_1} and \ref{ejemplo7_1} show that we do not
have analogues of Lemma 1.3 of \cite{greither}, namely if $L / K$ is
a cogalois extension, in the classical sense, and is such that $[L:
K] = p$, with $L = K(\alpha) $, $\alpha^{p} = a \in K$ and $L / K$
is separable and pure, then

(a) The group $\cogalois(L / K)$ has no elements of order $q \neq
p$, $q$ a prime number.

(b) The group $\cogalois(L / K)$ has no elements of order $p^{2}$.

\section{A bound for $\mid \cogalois(L/K)\mid$}\label{estimacion_para_cogalois}

In this section we establish an upper bound for the order of
$\cogalois(L / K)$. In what follows let $q = p^{\nu}$, and let $L/K$
be a radical extension. The following lemma will be needed in the
section.

\begin{lema}\label{auxiliar_raicescarlitz}
Let $K/k$ be a finite extension. Then there exists $M\in R_{T}$ such
that $\mu(K)=\Lambda_{M}$.
\end{lema}

\begin{proof}
For each $w\in\mu(K)$ there exist $D_{w},N_{w}\in R_{T}$ such that
$w=\lambda^{D_{w}}_{N_{w}}$. Canceling common factors we may assume
that the greatest common divisor of $D_{w}$ and $N_{w}$ is $1$.
Therefore $\Lambda_{N_{w}}\subseteq K$ and we obtain that $\mu(K)$
is finite. Thus there exist $N_{1},\ldots ,N_{r}\in R_{T}$ such that
$\Lambda_{N_{i}}\subseteq K$, for each $i=1,\ldots ,r$. Note that
$\mu(K)\subseteq\cup^{r}_{i=1}\Lambda_{N_{i}}$. Let $M\in R_{T}$ be
the least common multiple of $N_{1},\ldots, N_{r}$. Then
$\lambda_{M}\in K$. Since $\Lambda_{N_{i}}\subseteq\Lambda_{M}$ for
each $i=1,\ldots,r$, it follows that $\mu(K)=\Lambda_{M}$.
\end{proof}

\begin{comentario}\label{observacion_elemental_abeliano}
{\rm{If $L / K$ is a Galois and radical cyclotomic extension such
that $\mu(K) = \mu(L) $ then, since $L$ is radical, it is of the
form $K(\rho_{1}, \ldots, \rho_{t}) $, with $\rho^{M_{i}}_{i} =
a_{i} \in K$ for some $M_{i} \in R_{T}$. On the other hand the roots
of the polynomial $X^{M_{i}}-a_{i}$ are
$\{\rho_{i}+\lambda^{A}_{M_{i}}\}_{A\in R_{T}}$. Therefore
$\Gal(K(\rho_{i}) / K) \subseteq \Lambda_{M_ {i}}$. Thus $\Gal(K
(\rho_{i}) / K)$ is a $p$-elementary abelian group.

Since we have an injective map
\[
\Gal(L / K) \hookrightarrow \prod^{t}_{i = 1} \Gal(K (\rho_{i}) / K),
\]
it follows that $\Gal(L / K)$ is a $p$-elementary abelian
group.

For any field $L$ such that $L/k$ is finite, we have
$\mu(L)=\Lambda_{M}$ for some $M\in R_{T}$, by Lemma
\ref{auxiliar_raicescarlitz}. We define
\[
\grado(\mu (L)) = \grado(M).
\]
}}
\end{comentario}

\begin{lema}\label{auxiliar_coho_cogalois}
Let $ L / K $ be a finite radical cyclotomic Galois extension. Then
\[
B^{1}(G,\mu(L))\cong \mu(L)/\mu(K)
\]
as $R_{T}$-modules.
\end{lema}

\begin{proof}
We define $ \psi: \mu(L) \rightarrow B^{1}(G, \mu (L))$ as follows:
$\psi(u) = f_{u} $, see Equation (\ref{ec7_1}). We have $\psi(u + v)
= \psi(u) + \psi(v)$ and $\psi(u^{M})=f_{u^{M}}$ for $u,v\in
\mu(L)$, $M\in R_{T}$. Therefore $\psi$ is a homomorphism of
$R_{T}$-modules and it is surjective by the definition of $B^{1}(G,
\mu(L))$. Since $L/K$ is Galois extension we have
$\ker(\psi)=\mu(K)$ and the result follows.
\end{proof}

\begin{proposicion}\label{acotacion1}
Let $L / K$ be a finite radical cyclotomic Galois extension. If $\mu(L) =
\mu(K)$, then
\[
\mid\cogalois(L/K)\mid=q^{m\grado(\mu(L))}
\]
where $[L:K]=p^{m}$.
\end{proposicion}

\begin{proof}
From Remark \ref{observacion_elemental_abeliano} we have $\Gal(L /
K) \cong C^{m}_{p}$, for some $m \in {\mathbb{N}}$. Since $B^{1}(G,
\mu(L)) = \{0 \}$ and $H^{1}(G, \mu(L)) \cong {\rm Hom}(G, \mu(L))$,
from Proposition \ref{finitud_TC/L}, we obtain that
\[
\cogalois(L/K)\cong Z^{1}(G,\mu(L))/B^{1}(G,\mu(L))=
 H^{1}(G,\mu(L))\cong {\rm Hom}(G,\mu(L)).
\]

Furthermore $\mu(L) \cong C^{\nu\grado(\mu(L))}_{p}$.

On the other hand, if we denote by
${\mathfrak{L}}_{p}({\mathbb{F}}^{m}_{p},{\mathbb{F}}^{\nu\grado(\mu(L))}_{p})$
the set of linear transformations from ${\mathbb{F}}^{m}_{p}$ to
${\mathbb{F}}^{\nu\grado(\mu(L))}_{p}$ over ${\mathbb{F}}_{p}$ and
${\mathfrak{M}}_{m\times \nu{\rm deg}(\mu(L))}({\mathbb{F}}_{p})$
denotes the set matrices $m\times \nu{\rm deg}(\mu(L))$ with
coefficients in ${\mathbb{F}}_{p}$, then
\begin{align*}
{\rm Hom}(G,\mu(L)) &= {\rm Hom}(C^{m}_{p},C^{\nu{\rm deg}(\mu(L))}_{p}) 
={\mathfrak{L}}_{p}({\mathbb{F}}^{m}_{p},{\mathbb{F}}^{\nu{\rm deg}(\mu(L))}_{p})\\
&={\mathfrak{M}}_{m\times \nu{\rm deg}(\mu(L))}({\mathbb{F}}_{p})
\end{align*}

Thus $\mid {\rm Hom}(G,\mu(L))\mid= q^{m\grado(\mu(L))}$.
\end{proof}

\begin{ejemplo}\label{ejemplo_schultheis_12}
{\rm{Let $L$ be the splitting field of $f(X)=X^{T}-1 \in
k(\Lambda_{T})[X]$. From Example \ref{ejemplo_schultheis} we have
$\mid\cogalois(L/k(\Lambda_{T}))\mid=[L:k(\Lambda_{T})]=q=q^{m\grado(\mu(L))}$
as established in Proposition \ref{acotacion1}.\hfill $\square$}}
\end{ejemplo}

\begin{proposicion}\label{acotacion2}
Let $L / K$ be a finite radical cyclotomic Galois extension and assume that
$L = K(\mu (L))$. Then $\mid \cogalois(L/K) \mid\leq q^{m
\grado(\mu(L))}$, for some $m\in {\mathbb{N}}$.
\end{proposicion}

\begin{proof}
From Corollary \ref{cogalois_grado_p_pureza} we have $[L: K] = p^{m}
$ for some $m \in {\mathbb {N}}$, where $p=\caracteristica(K)$. The
proof is by induction on $ m $. Let $L / K$ be a finite radical cyclotomic
Galois extension, such that $L = K(\mu(L))$ and $[L: K] = p$. Then
$L / K$ is a cyclic extension of degree $p$. Let $M =
P^{\alpha_{1}}_{1} \cdots P^{\alpha_{r}}_{r} $ and $N =
P^{\beta_{1}}_{1} \cdots P^{\beta_{r}}_{r} $, with $1 \leq \beta_{i}
\leq \alpha_{i} $, $i = 1, \ldots, r$, such that
$\mu(L)=\Lambda_{M}$ and $\mu(K)=\Lambda_{N}$.

Let $G :=\Gal(L/K)= (\sigma)$. Since the action of Carlitz-Hayes commutes with
$\sigma$, we have $\sigma(\lambda_{M})=\lambda^{A}_{M}$ for some
$A\in R_{T}$. Since $\lambda_{M}\notin K$, $\sigma(\lambda_{M})\neq
\lambda_{M}$. Therefore $M \nmid (A-1)$. Let $M = ND$, so
$\lambda_{N} = \lambda^{D}_{M}$. We have
\[
\lambda_{N}=\sigma(\lambda_{N})=\sigma(\lambda^{D}_{M})
=(\sigma(\lambda_{M}))^{D}=\lambda^{AD}_{M}=\lambda^{A}_{N}.
\]
Hence $\lambda^{A-1}_{N} = 0 $ and $N \mid (A-1)$.

On the other hand we obtain
\begin{align*}
{\rm Tr}_{G}(\lambda_{M}) &=
\lambda_{M}+\lambda^{A}_{M}+\lambda^{A^{2}}_{M}+\cdots
+\lambda^{A^{p-1}}_{M}\\
&=\lambda^{1+A+A^{2}+\cdots +
A^{p-1}}_{M}=\lambda^{\frac{A^{p}-1}{A-1}}_{M}
=\lambda^{(A-1)^{p-1}}_{M},
\end{align*}
since $\frac{A^{p}-1}{A-1}=\frac{(A-1)^{p}}{A-1}$. 

Therefore ${\rm Tr}_{G}(\lambda_{M})\in K\cap
\Lambda_{M}=\Lambda_{N}$. Hence there exists $C \in R_{T}$ such that
$\lambda^{(A-1)^{p-1}}_{M} = \lambda^{C}_{N}$. Since $\sigma^{p} = 1
$ we obtain $\sigma^{p}(\lambda_{M}) = \lambda^{A ^ {p}}_{M} =
\lambda_{M} $, that is, $\lambda^{A^{p} -1}_{M} = 0$. Since $A^{p}
-1 = (A-1)^{p} $, we have that $M \mid (A-1)^{p} $.

We can write $A-1 = P^{\gamma_{1}}_{1} \cdots P^{\gamma_{r}}_{r}Q $
with $(Q,P_{1}\cdots P_{r})=1$. If $\beta_{i} < \alpha_{i} $ we have
$\lambda_{NP_{i}}\in L \setminus K$, furthermore 
$\sigma(\lambda_{NP_{i}})=\lambda^{A}_{NP_{i}}\neq
\lambda_{NP_{i}}$.
Thus $NP_{i}\nmid (A-1)$.

It follows that:

(i) $\gamma_{i_{0}} < \alpha_{i_{0}} $ for some $i_{0} \in \{1,
\ldots, r \}$ because $M\nmid (A-1)$

(ii) Since $N \mid (A-1)$ we have $\beta_{i} \leq \gamma_{i}$. 
Furthermore, if $\beta_i<\alpha_i$, since $NP_{i}\nmid (A-1)$ it follows that
$\beta_{i} = \gamma_{i}$.

(iii) Since $\lambda^{(A-1)^{p-1}}_{M}=\lambda^{C}_{N}$ we obtain
$\alpha_{i} - (p-1) \gamma_{i} \leq \beta_{i}$ for $1 \leq i \leq
r$.

(iv) Since $M\mid (A-1)^{p}$ we have $\alpha_{i}\leq p\gamma_{i}$ for
$1\leq i\leq r$.

Now ${\rm Tr}_{G}(\lambda^{B}_{M})=({\rm
Tr}(\lambda_{M}))^{B}=\lambda^{B(A-1)^{p-1}}_{M}$
for any $B\in R_T$.

Let $B = P^{\delta_{1}}_{1} \cdots P^{\delta_{r}}_{r}R$ with $(R,
P_{1} \cdots P_{r}) = 1$. Then
\begin{align*}
\lambda^{B}_{M}\in \ker\,{\rm Tr}_{G} &\Leftrightarrow \delta_{i}+(p-1)\gamma_{i}\geq\alpha_{i}\, \text{for each $i$}\\
&\Leftrightarrow \delta_{i}\geq 0 \,\text{and}\, \delta_{i}+(p-1)\gamma_{i}\geq\alpha_{i}\, \text{for each $i$}\\
&\Leftrightarrow \delta_{i}\geq{\rm
max}\{0,\alpha_{i}-(p-1)\gamma_{i}\} \text{ for each $i$.}
\end{align*}

Therefore $\ker\,{\rm Tr}_{G} = (\lambda^{B}_{M})$ with
$B=P^{\delta_{1}}_{1}\cdots P^{\delta_{r}}_{r}$ and $\delta_{i}={\rm
max}\{0,\alpha_{i}-(p-1)\gamma_{i}\}\leq \alpha_i$ for $1\leq i\leq r$. Thus
$(\lambda^{B}_{M})=(\lambda_{M^{\prime}})$, with
$M^{\prime}=P^{\mu_{1}}_{1}\cdots P^{\mu_{r}}_{r}$ where $\mu_{i}=
\alpha_{i}-\delta_{i}$ for $1\leq i\leq r$.

Furthermore $I_{G}(\lambda_{M}) = ((\sigma-1) \lambda_{M}) =
(\lambda^{A-1}_{M}) $, where $I_{G}: \mu(L) \rightarrow \mu(L)$ is
the homomorphism defined by $I_{G}(u) = \sigma(u)-u $. On the other
hand $I_{G} (\lambda_{M}) = (\lambda_{M ^{\prime \prime}})$, with
$M^{\prime \prime} = P^{\varphi_{1}}_{1} \cdots
P^{\varphi_{r}}_{r}$, where $\varphi_{i} = {\rm max} \{\alpha_{i} -
\gamma_{i}, 0 \}$, $1 \leq i \leq r$.

From (ii) we obtain that $\varphi_{i} = \alpha_{i} - \beta_{i} $ if
$\beta_{i} < \alpha_{i} $. If $\alpha_{i}=\beta_{i}$ from (ii)
follows that $\alpha_{i} - \gamma_{i} \leq 0$. Therefore
$\varphi_{i} = \alpha_{i} - \beta_{i}$. Thus

\begin{gather*}
\mid
H^{1}(G,\mu(L))\mid = \frac{\mid
(\lambda_{M^{\prime}})\mid}{\mid(\lambda_{M^{\prime\prime}})\mid}=\mid(\lambda_{M^{\prime\prime\prime}})\mid\\
\intertext{with $M^{\prime\prime\prime}=P^{\varepsilon_{1}}_{1}\cdots
P^{\varepsilon_{r}}_{r}$ where}
\varepsilon_{i}=\mu_{i}-\varphi_{i}=\alpha_{i}-\delta_{i}-
(\alpha_{i}-\beta_{i})=\beta_i-\delta_i,\quad 1\leq i\leq r.
\end{gather*}

Obviously, $\varepsilon_{i} \leq \beta_{i}$.
Therefore
\[
\mid H^{1}(G,\mu(L))\mid=q^{{\rm deg}M^{\prime\prime\prime}}\leq q^{{\rm deg}N}.
\]
Combining the above inequality and
 Lemma \ref{auxiliar_coho_cogalois},we obtain
\begin{align*}
\mid\cogalois(L/K)\mid &= \mid H^{1}(G,\mu(L))\mid\mid
B^{1}(G,\mu(L))\mid = \mid H^{1}(G,\mu(L))\mid\frac{\mid \mu(L)
\mid}{\mid
\mu(K)\mid}\\
&=q^{{\rm deg}M^{\prime\prime\prime}}q^{{\rm deg}M-{\rm deg}N} \leq
q^{{\rm deg}M}= q^{{\rm deg}(\mu(L))}.
\end{align*}

Now we suppose that $[L:K]=p^m$ with $m\geq 2$. Let
$H$ be a subgroup of $G$ of order $p^{m-1}$. Let $E = L^{H}$. Then
$K \subseteq E \subseteq L$. We have $[E: K] = p$, $[L: E] = p^{m-1}
$ and $L = E(\mu (L))$.

If $E / K$ were not a radical cyclotomic extension, we would have
$\cogalois(E / K) = \{0 \}$, since otherwise there would exist a
nonzero element $\overline{\alpha} \in \cogalois(E / K)$. In
particular $\alpha \notin K$. It follows that $E = K (\alpha)$, but
this implies that $E / K$ is a radical cyclotomic extension, which
is absurd.

Thus $E / K$ is a radical extension cyclotomic and we consider two
cases:

(i) $\mu(E)\neq\mu(K)$ and

(ii) $\mu(E)=\mu(K).$

In the first case, from what we proved in the case $[E:K]=p$ we
obtain
\[
\mid\cogalois(E/K)\mid\leq q^{\grado(\mu(E))}.
\]

In case (ii), from Proposition \ref{acotacion1} we have $\mid
\cogalois(E / K) \mid = q^{\grado(\mu(E))}.$ Thus, in any case
\[
\mid\cogalois(E/K)\mid\leq q^{\grado(\mu(E))}\leq q^{\grado(\mu(L))}.
\]

Since $L = E(\mu(L))$ and $[L: E] = p^{m-1}$, by the induction
hypothesis we have $\mid \cogalois(L / E) \mid \leq q^{(m-1)
\grado(\mu(L))}$. Therefore from the exact sequence of
$R_{T}$-modules
\[
0\rightarrow {\rm cog}(E/K)\rightarrow {\rm cog}(L/K)\rightarrow{\rm cog}(L/E)
\]
it follows that $\mid \cogalois(L/K)\mid\leq\mid
\cogalois(E/K)\mid\mid \cogalois(L/E)\mid\leq q^{m\grado(\mu(L))}$.
\end{proof}

From the proof of Proposition \ref{acotacion2} for the case $m = 1$,
we have the following corollary.

\begin{corolario}
Let $L / K$ be a cyclic extension of degree $p$ with $L = K(\alpha)$
such that $\alpha \in \cogalois(L/K)$. Then $L/K$ is a radical
cyclotomic extension with $\mid\cogalois(L/K)\mid=p^{\nu t}$, where
$q=p^{\nu}$ and
\[
t = \left\lbrace
           \begin{array}{c l}
 \grado(\mu(L))-\grado(\mu(K))+\grado(\frac{B-1}{C}) & \text{if $\mu(L)\neq\mu(K)$},\\
              \grado(\mu(L)) & \text{if $\mu(L)=\mu(K)$},
           \end{array}
         \right.
\]
where $\sigma(\lambda_{M})=\lambda^{A}_{M}$, $B-1:=
\gcd(A-1,M)$ and $C$ is
the polynomial of minimum degree such that $C \mid (B-1)$ and $M
\mid C (B-1)^{p-1}$. \hfill $\square$
\end{corolario}

\begin{proposicion}\label{acotacion3}
Let $L/K$ be a finite radical cyclotomic Galois extension. Then
\[
\mid\cogalois(L/K)\mid\leq q^{m\grado(\mu(L))}
\]
where $[L:K]=p^{m}$.
\end{proposicion}

\begin{proof}
Let $E=K(\mu(L))$ with $K\subseteq E\subseteq L$. Then
\[
\mid \cogalois(L/K)\mid\leq\mid \cogalois(E/K)\mid\mid
 \cogalois(L/E)\mid\leq q^{m{\rm deg}(\mu(L))}
\]
by Propositions \ref{acotacion1} and \ref{acotacion2}.
\end{proof}

The following example shows that the inequality in Proposition
\ref{acotacion3} may be strict.

\begin{ejemplo} \label{ejemplo_no_se_alcanza_cota}
{\rm{Let $L=k(\Lambda_{P^{2p-1}})$, with $P\in R_{T}$ irreducible,
and $\sigma=1+P^{2}\in \Gal(k(\Lambda_{P^{2p-1}})/k)$. We have
\[
\sigma(\lambda_{P^{2p-1}})=\lambda_{P^{2p-1}}+\lambda_{P^{2p-3}}\neq 
\lambda_{P^{2p-1}}.
\]
Thus $\sigma\neq 1$.

On the other hand $\sigma^{p}= (1+P^{2})^{p}=1+P^{2p}$. Hence
\[
\sigma^{p}(\lambda_{P^{2p-1}})=
\lambda_{P^{2p-1}}+\lambda^{P^{2p}}_{P^{2p-1}}
=\lambda_{P^{2p-1}}.
\]
Therefore $\sigma^{p}=1$, and the order of $\sigma$ is $p$.

Let $E = L^{(\sigma)}$. Then $[L:E]=p$ and $L/E$ is a radical
cyclotomic extension. We have $\sigma(\lambda^{M}_{P^{2p-1}})=
\lambda^{M}_{P^{2p-1}}+\lambda^{M}_{P^{2p-3}}=\lambda^M_{P^{2p-1}}$ if
and only if the exponent which appears in $P$ in the decomposition
of $M$ as product of irreducible polynomials is greater than or
equal to $2p-3$. In this case $\lambda^{P^{2p-3}}_{P^{2p-1}}=
\lambda_{P^{2}}\in E$. Note that $\lambda_{P^{3}}\notin E$.
Therefore $\mu(E)=\Lambda_{P^{2}}$. Further $\mu(L)=
\Lambda_{P^{2p-1}}$.

Let $N_{\mu(L)}$ be the trace map from $L$ to $E$, that is
$N_{\mu(L)}:=\sum_{i=0}^{p-1}\sigma^i$. We have
$N_{\mu(L)}(\lambda^{M}_{P^{2p-1}})=
\lambda^{M(\frac{(1+P^{2})^{p}-1 }{(1+P^{2})-1 } )
}_{P^{2p-1}}=\lambda^{MP^{2p-2}}_{P^{2p-1}}=\lambda^{M}_{P}=0$ if
and only if $P$ divides $M$. Therefore ${\rm Ker}\,
N_{\mu(L)}=(\lambda^{P}_{P^{2p-1}})=\Lambda_{P^{2p-2}}$.

Let $G:=\Gal(L/E)$ and $I_G=\langle \sigma\rangle$. Then
$I_{G}(\mu(L))=(\sigma(\lambda_{P^{2p-1}})-\lambda_{P^{2p-1}})=
(\lambda_{P^{2p-3}})=\Lambda_{P^{2p-3}}$. Therefore
\begin{gather*}
\mid{\rm cog}(L/E)\mid=\mid H^{1}(G,\mu(L))\mid\frac{\mid 
\mu(L) \mid }{\mid \mu(E)  \mid}=\frac{\mid
\Lambda_{ P^{2p-2}} \mid }{\mid\Lambda_{P^{2p-3}} 
 \mid}\frac{\mid\Lambda_{P^{2p-1}}  \mid}{\mid\Lambda_{P^{2}}  
  \mid}= q^{d(2p-2)}\\
\intertext{where $d=\grado(P)$. Since $m=1$, we have}
\mid\cogalois(L/E)\mid= q^{d(2p-2)}<
q^{d(2p-1)}=q^{m\grado(\mu(L))}.
\end{gather*}
}}
\end{ejemplo}

\begin{teorema}\label{cota_superior_cog}
Let $L / K$ be a finite radical cyclotomic extension. Then if
$\widetilde{L}$ is the Galois closure of $L$, we have
\[
\mid{\rm cog}(L/K)\mid\leq q^{m\grado(\mu(\widetilde{L}))}
\]
where $[\widetilde{L}:K]= p^{m}$.
\end{teorema}

\begin{proof}
Let $G = \Gal(\widetilde{L} / K) = HN$ with $H$ a normal subgroup of
$G$ and $N$ a $p$-Sylow subgroup of $G$. Let $F = \widetilde{L}^{H}
$. We can assume that $F = L$ by changing $H$ by a conjugate. We
obtain the diagram
\[
 \xymatrix{L \ar@{-}[r]^{H}\ar@{-}[d]
          & \widetilde{L} \ar@{-}[d]^{N} \\
     K \ar@{-}[r] & E }
\]

Let $\alpha$ be a nonzero element of $\cogalois(L/K)$. There exists
$N \in R_{T}$ such that $\alpha^{N} = a \in K$. Since $\alpha \in
\widetilde{L}$, $\alpha^{N} \in K \subseteq E$, that is, $\alpha \in
\cogalois(\widetilde{L} / E)$. If $\alpha = 0$ in
$\cogalois(\widetilde{L} / E)$, we would have $\alpha \in E \cap L =
K$ so that $\alpha = 0$ in $\cogalois(L / K)$, a contradiction. Then
$\cogalois(L/K)\subseteq \cogalois(\widetilde{L}/E)$.

Therefore $\mid\cogalois(L/K)\mid\leq
\mid\cogalois(\widetilde{L}/E)\mid\leq
q^{m\grado(\mu(\widetilde{L}))}$.
\end{proof}


\begin{thebibliography}{25}

\bibitem{Albu} T. Albu, {\it Cogalois theory}, Marcel Dekker,
New York, 2003.

\bibitem{barrera1} F. Barrera-Mora, M. Rzedowski-Calder\'on, and G. Villa-Salvador,
{\it On cogalois extensions}, J. Pure Appl. Algebra {\bf 76} (1991), 1--11.

\bibitem{barrera2} F. Barrera-Mora and W. Yslas-Velez, {\it Some results on radical
extensions}, J. of Algebra {\bf 162} (1993), 295--301.

\bibitem{greither} C. Greither and D.K. Harrison, {\it A Galois correspondence for radical extensions of fields}, J. Pure Appl. Algebra {\bf 43} (1986), 257--270.

\bibitem{hall} M. Hall, Jr. {\it The theory of groups}, Macmillan, New York, 1959.

\bibitem{hayes} D.R. Hayes, {\it Explicit class field theory for rational
function fields}, Trans. Amer. Math. Soc. {\bf 189} (1974), 77--91.

\bibitem{hsu} Chih-Nung Hsu, {\it On Artin conjecture for the Carlitz
module}, Compositio Mathematica {\bf 106} (1997), 247--266.

\bibitem{Lam} P. Lam-Estrada and G. D. Villa-Salvador, {\it Some remarks on the theory of cyclotomic function
fields}, Rocky Mountain Journal of Mathematics {\bf 31}, No. 2 (2001),
483--502.

\bibitem{Schultheis} F. Schultheis, {\it Carlitz-Kummer Function Fields}, Journal of Number Theory {\bf 36} (1990), 133--144.

\bibitem{villa_salvador} G. D. Villa Salvador, {\it Topics in the Theory of
Algebraic Function Fields}, Birkh\"auser, Boston, 2006.

\end{thebibliography}
\end{document}